\documentclass[reqno]{amsart}
\usepackage[utf8]{inputenc}
\usepackage{float}
\usepackage{amsmath}
\usepackage{graphicx}
\usepackage{latexsym}
\usepackage{amsfonts}
\usepackage{amssymb}
\usepackage{color}
\theoremstyle{plain}
\newtheorem{theorem}{Theorem}
\newtheorem{corollary}[theorem]{Corollary}
\newtheorem{lemma}[theorem]{Lemma}
\newtheorem{proposition}[theorem]{Proposition}
\theoremstyle{definition}

\newtheorem{remark}[theorem]{Remark}
\newtheorem*{remark*}{Remark}

\begin{document}
\title[Persistence of one-dimensional AR(1)-sequences]
{Persistence of one-dimensional AR(1)-sequences}

\author[Hinrichs]{G\"unter Hinrichs} 
\address{Institut f\"ur Mathematik, Universit\"at Augsburg, 86135 Augsburg, Germany}
\email{Guenter.Hinrichs@math.uni-augsburg.de}

\author[Kolb]{Martin Kolb}
\address{Institut f\"ur Mathematik, Universit\"at Paderborn,
        33098 Paderborn, Germany}
\email{kolb@math.uni-paderborn.de}

\author[Wachtel]{Vitali Wachtel} 
\address{Institut f\"ur Mathematik, Universit\"at Augsburg, 86135 Augsburg, Germany}
\email{vitali.wachtel@mathematik.uni-augsburg.de}

\begin{abstract}
For a class of one-dimensional autoregressive sequences $(X_n)$ we consider the tail behaviour of the 
stopping time $T_0=\min \lbrace n\geq 1: X_n\leq 0 \rbrace$. We discuss existing general analytical
approaches to this and related problems and propose a new one, which is based on a renewal-type
decomposition for the moment generating function of $T_0$ and on the analytical Fredholm alternative.
Using this method, we show that $\mathbb{P}_x(T_0=n)\sim V(x)R_0^n$ for some $0<R_0<1$ and a positive 
$R_0$-harmonic function $V$. Further we prove that our conditions on the tail behaviour of the
innovations are sharp in the sense that fatter tails produce non-exponential decay factors.  
\end{abstract}
\maketitle
%%%%%%%%%%%%%%%%%%%%%%%%%%%%%%%%%%%%%%%%%%%%%%%%%%%%%%%%%%%%%%%%%%%%%%%%%%%%%%%%%%%%%%%%%%%%%%%%%%%%%%%%
%%%%%%%%%%%%%%%%%%%%%%%%%%%%%%%%%%%%%%%%%%%%%%%%%%%%%%%%%%%%%%%%%%%%%%%%%%%%%%%%%%%%%%%%%%%%%%%%%%%%%%%%
%%%%%%%%%%%%%%%%%%%%%%%%%%%%%%%%%%%%%%%%%%%%%%%%%%%%%%%%%%%%%%%%%%%%%%%%%%%%%%%%%%%%%%%%%%%%%%%%%%%%%%%%
\section{Introduction and setting}
The analysis of first hitting times of subsets of the state space by a Markov chain
$(X_n)_{n\geq 0}$ is a subject with a long history, but still many recent contributions. For
many applications it is important to gain precise control of the tail behaviour of hitting times. One of
the aims of this work is to demonstrate the usefulness of the combination of analytic and probabilistic
techniques using the example of autoregressive processes.

Let $\xi_k$ be independent, identically
distributed random variables and let $a\in(0,1)$ be a fixed constant. An AR$(1)$-sequence is defined by 
\begin{equation}
\label{AR.def}
X_n=aX_{n-1}+\xi_n,\quad n\geq 1,
\end{equation}
where the starting point $X_0$ of this process may be either deterministic or
distributed according to any probabilistic measure $\nu$. If $X_0=x$ then
we write $\mathbb{P}_x$ for the distribution of the process and if $X_0$ is
distributed according to $\nu$ then we shall write $\mathbb{P}_\nu$ for the
distribution of the process. The sequence $(X_n)_{n\in\mathbb{N}_0}$ defines a 
Markov chain with state space $\mathbb{R}$, whose properties have been analysed
in a great number of papers and we only refer to some of the more recent contributions, such as
\cite{AM17}, \cite{AB11}, \cite{Christensen12}, \cite{Llaralde}, \cite{Nov09}, \cite{NK08}.
Closest to our present contribution and the main stimulus for the present paper is the recent work \cite{AM17},
where persistence probabilities for the process \eqref{AR.def} and its multidimensional versions have
been studied. The focus of \cite{AM17} has been on deriving the existence and basic properties such as
positivity and monotonicity of the persistence exponents 
\begin{displaymath}
\lambda_a:=\lim_{n \rightarrow \infty}\frac{1}{n}\log \mathbb{P}\bigl(X_0>0,X_1>0,\dots,X_n>0\bigr)
\end{displaymath}
for rather general Markov chains (including multidimensional cases) and its calculation
for some very specific chains. For our purposes, let us define 
$$
T_0:=\min\{k \geq 1:\ X_k\leq 0\},
$$
i.e. the first time at which the process becomes negative. Similar to \cite{AM17}, we are going to study the tail behaviour of this stopping time, but, in contrast to \cite{AM17}, we are aiming at precise instead of rough asymptotic results. This means we aim to find the precise, without logarithmic scaling, asymptotics of 
\begin{displaymath}
n \mapsto \mathbb{P}_x(T_0>n)
\end{displaymath}
as $n\rightarrow \infty$. Under natural and in some sense minimal conditions we aim to show that the tail of the stopping time $T_0$ has an exactly exponential decay. We will focus on the one-dimensional situation.

Denoting by $P(x,dy)$ the transition probability of the Markov chain
$(X_n)_{n\ge0}$ we observe that, for every $x>0$,
\begin{equation*}
\begin{split}
\mathbb{P}_x &\bigl(X_1>0,X_2>0,\ldots ,X_n>0\bigr)=\\
&\quad \quad\int_{(0,\infty)}P(x,dy_1)\int_{(0,\infty)}P(y_1,dy_2)\dots\int_{(0,\infty)}P(y_{n-1},(0,\infty))
\end{split}
\end{equation*}
and therefore the probability $\mathbb{P}_x \bigl(X_1>0,X_2>0,\ldots ,X_n>0\bigr)$ can be interpreted
as the $n$-th power of the total mass of the substochastic transition kernel given 
by
\begin{displaymath}
P_+(x,A):=\mathbf{1}_{(0,\infty)}(x)P(x,A \cap (0,\infty)).
\end{displaymath}
From the Gelfand formula for the spectral radius it is tempting to connect $\lambda_a$ to the spectral
radius of some operator induced by $P_+$. From an operator theoretic perspective the problem consists in the
fact that, due to the unboundednes of the state space $(0,\infty)$, the substochastic kernel $P_+$ is usually
not a (quasi-)compact operator on the standard Banach spaces of continuous or $p$-th power integrable functions.
One way out is to find better adapted Banach spaces, which is often possible. A different strategy which we are going to present as well consists in analysing the behaviour of the Laplace transform
\begin{displaymath}
\lambda \mapsto \mathbb{E}_x\bigl[e^{\lambda T_0}\bigr]
\end{displaymath}
near the critical line and in this respect is classical. In fact, similar arguments appear in the investigation of other large time problems in the theory of stochastic processes such as \cite{L95} and \cite{SW75}. In order to deduce the required properties we will show that $\mathbb{E}_x\bigl[e^{\lambda T_0}\bigr]$ satisfies a suitable renewal equation and study some operator theoretic properties of the corresponding transition operator.
This leads to a meromorphic representation for the function $\mathbb{E}_x\bigl[e^{z T_0}\bigr]$. The final step consists in showing 
that all singularities near the critical line $\{z:\, \Re z=\lambda_a\}$ are simple poles and in the subsequent application of the Wiener-Ikehara theorem. We want to emphasize that related results have been recently derived in \cite{CV17} using completely different methods and we will comment on connections below.

All our results will be valid for any stopping time
$$
T_r:=\min\{k \geq 1:\ X_k\leq r\},\quad r\in\mathbb{R}.
$$
This is immediate from the observation that the sequence $X_n^{(r)}:=X_n-r$, $n\geq0$
satisfies \eqref{AR.def} with innovations $\xi_n^{(r)}:=\xi_n-(1-a)r$, $n\ge1$.

We want to stress at this point that even though we exclusively deal with the one-dimensional situation in this work, the approaches we present are more generally applicable, also in multidimensional situations and to processes of different type. As it seems that our analytic approaches are not well known in the probabilistic literature, they are at least rarely used in standard literature, we hope that the present contribution also serves as template on powerful analytic methods for persistence and quasistationary problems.

Several authors, see \cite{Christensen12,Llaralde,Nov09}, have obtained exact expressions for the Laplace transform in the case when the distribution of the innovations is related to the exponential distribution. Unfortunately, the expressions are very complicated, and it is not clear how to invert them or how to use them for deriving the tail asymptotics for $T_0$. Moreover, it is not clear whether one can obtain an explicit expression for $\lambda_a$. If, for example, the innovations $\xi_n$ have density $e^{-\mu|x|}/(2\mu)$ then, as it has been shown in~\cite{Llaralde},
$$
\mathbb{E}_0 s^{T_0}=\frac{s(as,a^2)_\infty}{(as,a^2)_\infty+(s,a^2)_\infty},
$$
where
$(u,q)_\infty=\prod_{k=0}^\infty(1-uq^k)$. Therefore, $e^{\lambda_a}$ is the minimal positive solution to the equation $(as,a^2)_\infty+(s,a^2)_\infty=0$. It is obvious that
this solution lies between $1$ and $a^{-1}$, but an explicit expression is not accessible. In order to analyse the tail behaviour of $T_0$ we have to take into account
all singularities of this function on the circle of radius $e^{\lambda_a}$, but this information is also rather hard to extract from the exact expression. Expressions in
\cite{Christensen12} and in \cite{Nov09} are even more complicated. Summarising, all known explicit expressions for the Laplace transform of $T_0$ do not seem to provide any useful information on the asymptotic properties of $T_0$. 

The structure of the paper is the following. In Section \ref{s:rough} we prove under rather general assumptions the existence and positivity of the decay rate 
\begin{displaymath}
\lambda_a:=-\lim_{n\rightarrow \infty}\frac{1}{n}\log \mathbb{P}_x\bigl(T_0>n\bigr).
\end{displaymath}
In Section \ref{s:boundedsupport} we consider the situation where the distribution of the innovations has bounded support. In this case, one can use results in \cite{CV16} in order to show that the tails have a precise exponential decay. For unbounded innovations this approach is no longer possible.

In order to be able to cover also unbounded innovations we introduce in Section \ref{s:FuAn} a different functional analytic approach relying on the concept of a quasicompact operator and the associated spectral theory.

Section \ref{s:renewal} presents another way to deal with unbounded innovations, which is based on a renewal argument in combination with some basic operator theoretic arguments applied to the renewal operator. 

Section \ref{s:reg} deals with the situation of regularly varying case, where the foregoing theory is not applicable.

In Section \ref{s:dis}, the essential ingredients of our techniques and their applicability to different problems are discussed.
%%%%%%%%%%%%%%%%%%%%%%%%%%%%%%%%%%%%%%%%%%%%%%%%%%%%%%%%%%%%%%%%%%%%%%%%%%%%%%%%%%%%%%%%%%%%%%%%%%%%%%%%
%%%%%%%%%%%%%%%%%%%%%%%%%%%%%%%%%%%%%%%%%%%%%%%%%%%%%%%%%%%%%%%%%%%%%%%%%%%%%%%%%%%%%%%%%%%%%%%%%%%%%%%%
\section{Rough asymptotics for persistence probabilities}\label{s:rough}
In this section we  establish a general result concerning the exponential
decay of the tails of the hitting time $T_0$. Results of this type will be an essential ingredient for the further investigation of more detailed properties of the first hitting time $T_0$. More precisely, we shall derive
asymptotics for the logarithmically scaled tail of $T_0$. In contrast to \cite{AM17} we study the one-dimensional situation, but there we are able to work under much weaker hypotheses. 

Let us first recall some basic properties of the Markov process $(X_n)_{n\in\mathbb{N}_0}$. If $\mathbb{E}\log(1+|\xi_1|)<\infty$ then $X_n$ converges weakly to the distribution of the series
$$
X_\infty:=\sum_{k=1}^\infty a^{k-1}\xi_k.
$$
This is immediate from the following expression for $X_n$:
\begin{equation}
\label{eq.1}
X_n=a^nX_0+a^{n-1}\xi_1+a^{n-2}\xi_2+\ldots+\xi_n.
\end{equation}
Let $\pi(dx)$ denote the distribution of $X_\infty$. This distribution is stationary
for the Markov chain $X_n$, that is,
\begin{equation}
\label{eq.2}
\mathbb{P}_\pi(X_n\in dx)=\pi(dx),\quad n\geq1.
\end{equation}

\begin{theorem}
\label{thm.log}
Assume that the innovations $(\xi_n)_{n\in\mathbb{N}}$ satisfy
\begin{displaymath}
\mathbb{E}\log(1+|\xi_1|)<\infty,\, \mathbb{E}(\xi_1^+)^\delta<\infty \text{ for some }\delta>0
\quad\text{and}\quad
\mathbb{P}(\xi_1>0)\mathbb{P}(\xi_1<0)>0\,.
\end{displaymath}
Then, for every $a\in(0,1)$, 
\begin{equation}
\label{thm.log.1}
-\lim_{n\rightarrow\infty}\frac{1}{n}\log\mathbb{P}_x\bigl( T_0>n\bigr)
=\lambda_a\in (0,\infty),\quad x\in(0,\infty).
\end{equation}
Furthermore, if the distribution of the innovations satisfies
\begin{equation}
\label{thm.log.2}
\lim_{x\to\infty}\frac{\log\mathbb{P}(\xi_1>x)}{\log x}=0,
\end{equation}
then 
\begin{equation}
\label{thm.log.3}
-\lim_{n\rightarrow\infty}\frac{1}{n}\log\mathbb{P}_x\bigl( T_0>n\bigr)
=0,\quad x\in(0,\infty).
\end{equation}
\end{theorem}
If $\mathbb{P}(\xi_1<0)=0$ then $T_0=\infty$ almost surely. Furthermore, the
assumption $\mathbb{P}(\xi_1>0)>0$ is imposed to avoid a trivial situation
when the chain $X_n$ is monotone decreasing before hitting negative numbers.
As it has been mentioned at the beginning of this section, the existence of the
logarithmic moment yields the ergodicity of $X_n$. Finally, the finiteness of
$\mathbb{E}(\xi_1^+)^\delta$ is needed for the positivity of $\lambda_a$ only.

Existence and finiteness of the limit $\lambda_a$ follow also from Theorem 2.3 in
\cite{AM17} in a multidimensional situation at least under the condition that an
exponential moment exists. We use a weaker moment assumption
$\mathbb{E}(\xi_1^+)^\delta<\infty$. Since the finiteness of
$\mathbb{E}(\xi_1^+)^\delta$ implies that 
$$
\limsup_{x\to\infty}\frac{\log\mathbb{P}(\xi_1>x)}{\log x}\le-\delta,
$$
we conclude that the moment assumption $\mathbb{E}(\xi_1^+)^\delta<\infty$ is
optimal for the positivity of $\lambda_a$.

Relation \eqref{thm.log.3} is a simple consequence of relation (2.4) in \cite{AM17}.
It should be noted that the proof of (2.4) does not use the assumption that the
innovations are normally distributed and, consequently, we may use (2.4) in the proof
of \eqref{thm.log.3}.

\begin{proof}[Proof of Theorem~\ref{thm.log}]
It follows from \eqref{eq.1} that, for $0\le x\le y$,
\begin{align}
\label{eq.3}
\nonumber
\mathbb{P}_x(T_0>n)
&=\mathbb{P}\left(\min_{k\leq n}\left(xa^k+\sum_{j=1}^ka^{k-j}\xi_j\right)>0\right)\\
&\le \mathbb{P}\left(\min_{k\leq n}\left(ya^k+\sum_{j=1}^ka^{k-j}\xi_j\right)>0\right)
=\mathbb{P}_y(T_0>n).
\end{align}
Besides $T_0$ we consider a slightly modified stopping time
$$
\widetilde{T}_0:=\min\{k\ge0:X_k\le0\}.
$$

The monotonicity property \eqref{eq.3} implies that, for every $x>0$,
\begin{align*}
\mathbb{P}_\pi(\widetilde{T}_0>n)=\int_0^\infty\pi(dy)\mathbb{P}_y(T_0>n)
&\ge\int_x^\infty\pi(dy)\mathbb{P}_y(T_0>n)\\
&\ge \pi[x,\infty)\mathbb{P}_x(T_0>n).
\end{align*}
In other words,
\begin{equation}
\label{eq.4}
\mathbb{P}_x(T_0>n)\le \frac{1}{\pi[x,\infty)}\mathbb{P}_\pi(T_0>n). 
\end{equation}

Next, multiplying \eqref{eq.1} by $a^{-n}$, we have
$$
a^{-n}X_n=X_0+\sum_{j=1}^n a^{-j}\xi_j=:X_0+S_n
$$
and
$$
\left\{T_0>n\right\}=\left\{X_0+\min_{k\le n}S_k>0\right\}.
$$
Since $X_0+S_k$ are sums of independent random variables, we may apply FKG inequality
for product spaces, which gives the estimate
\begin{align}
\label{eq.5}
\nonumber
\mathbb{P}(X_n>y|T_0>n)
&=\mathbb{P}\left(X_0+S_n>a^{-n}y\Big| X_0+\min_{k\le n}S_k>0\right)\\
&\ge \mathbb{P}\left(X_0+S_n>a^{-n}y\right)=\mathbb{P}(X_n>y),
\end{align}
which holds for every distribution of $X_0$. Applying this inequality to the 
stationary process, we obtain
\begin{align*}
\mathbb{P}_\pi(\widetilde{T}_0>n+m)&=\int_0^\infty\mathbb{P}_\pi(X_n\in dy,\widetilde{T}_0>n)\mathbb{P}_y(T_0>m)\\
&=\mathbb{P}_\pi(\widetilde{T}_0>n)\int_0^\infty\mathbb{P}_\pi(X_n\in dy|\widetilde{T}_0>n)\mathbb{P}_y(T_0>m)\\
&\ge \mathbb{P}_\pi(\widetilde{T}_0>n)\int_0^\infty\mathbb{P}_\pi(X_n\in dy)\mathbb{P}_y(T_0>m)\\
&=\mathbb{P}_\pi(\widetilde{T}_0>n)\mathbb{P}_\pi(\widetilde{T}_0>m).
\end{align*}
Then, by the Fekete lemma,
\begin{equation}
\label{eq.6}
-\lim_{n\to\infty} \frac{1}{n}\log\mathbb{P}_\pi(\widetilde{T}_0>n)=:\lambda_a(\pi)\in[0,\infty).
\end{equation}
By the same argument, for every fixed $x$ we have
\begin{align*}
\mathbb{P}_x(T_0>n+m)
&\ge\mathbb{P}_x(T_0>n)\int_0^\infty\mathbb{P}_x(X_n\in dy)\mathbb{P}_y(T_0>m)\\
&\ge\mathbb{P}_x(T_0>n)\int_x^\infty\mathbb{P}_x(X_n\in dy)\mathbb{P}_y(T_0>m).
\end{align*}
Using now the monotonicity property \eqref{eq.3}, we get
\begin{equation}
\label{eq.6a}
\mathbb{P}_x(T_0>n+m)\ge \mathbb{P}_x(T_0>n)\mathbb{P}_x(X_n\ge x)\mathbb{P}_x(T_0>m).
\end{equation}
If $x$ is such that $\pi[x,\infty)>0$ then $\mathbb{P}_x(X_n\ge x)\to\pi[x,\infty)$
and we may apply again the Fekete lemma:
\begin{equation}
\label{eq.7}
-\lim_{n\to\infty} \frac{1}{n}\log\mathbb{P}_x(T_0>n)=:\lambda_a(x)\in[0,\infty).
\end{equation}
By \eqref{eq.3}, $\lambda_a(x)$ is decreasing in $x$. Moreover, from \eqref{eq.4}, 
\eqref{eq.6} and \eqref{eq.7}, we infer that
$$
\lambda_a(x)\ge\lambda_a(\pi)\quad\text{for all }x\text{ such that }\pi[x,\infty)>0.
$$
If $\pi[x,\infty)>0$, then there exist $\varepsilon>0$ and $m_0$ such that
$\mathbb{P}_x(X_{m_0}>x+\varepsilon)>0$. Then, using the Markov property at time
$m_0$ and the monotonicity of $P_y(T_0>n)$, we conclude that
$\lambda_a(x)=\lambda_a(y)$ for all $y\in[x,x+\varepsilon]$. As a result, we have
$$
\lambda_a(x)=\lambda_a(\pi)=:\lambda_a \quad
\text{for all }x\text{ such that }\pi[x,\infty)>0.
$$
According to Theorem 1 in \cite{NK08}, the assumption
$\mathbf{E}(\xi_1^+)^\delta<\infty$ for some $\delta>0$ implies that $\lambda_a>0$.
Thus, we have the same exponential rate for all starting points in the support of
the measure $\pi$.

If $\mathbb{P}(\xi_1>x)>0$ for all $x>0$, then we have the logarithmic asymptotic behaviour
for all positive starting points.

Consider now the case when innovations are bounded from above. Let $R$ denote the
essential supremum of $\xi_1$, that is, $\mathbb{P}(\xi_1\le R)=1$ and
$\mathbb{P}(\xi_1>R-\varepsilon)>0$ for every $\varepsilon>0$. It is easy to see
that $\pi([x,\infty))>0$ for every $x\in(0,R/(1-a))$ and $\pi([x,\infty))=0$ for each
$x>R/(1-a)$. If the starting point $x$ is greater than $R_*:=R/(1-a)$ then the 
sequence $X_n$ decreases before it hits $(0,R_*)$ and there exists
$m=m(x,\varepsilon)$ such that $X_m\le R_*+\varepsilon$. Further, for every starting
point in $[R_*,R_*+\varepsilon)$ the hitting time of $(-\infty,R_*)$ is stochastically
bounded by a geometric random variable with parameter $1-p_\varepsilon$, where
$p_\varepsilon:=\mathbb{P}(\xi_1>R-\varepsilon)>0$. Since we know the exponent for
all starting points in $(0,R_*)$, we conclude that
$$
-\lim_{n\to\infty}\frac{1}{n}\log\mathbb{P}_{x}(T_0>n)
=\min\{\lambda_a,\log(1/p_\varepsilon)\},\quad x\ge R_*.
$$
Letting now $\varepsilon\to0$, we obtain
$$
-\lim_{n\to\infty}\frac{1}{n}\log\mathbb{P}_{x}(T_0>n)
=\min\{\lambda_a,\log(1/p)\},\quad x\ge R_*,
$$
where $p=\mathbb{P}(\xi_1=R)$. Now, noting that
$$
\mathbb{P}_x(T_0>n)\ge \mathbb{P}(\xi_1=\xi_2=\ldots=\xi_n=R)=p^n,\quad x>0,
$$
we infer that $\lambda_a\le \log(1/p)$. Consequently,
$$
-\lim_{n\to\infty}\frac{1}{n}\log\mathbb{P}_{x}(T_0>n)
=\lambda_a,\quad x\ge R_*.
$$
This completes the proof of \eqref{thm.log.1}.

It is clear that the distribution of $\xi_1$ is a uniform minorant for distributions
of $X_1$, that is,
$$
\mathbb{P}_x(X_1>y)\ge \mathbb{P}(\xi_1>y),\quad\text{for all }x,y>0.
$$
Therefore,
$$
\mathbb{P}_x(T_0>n)\ge \mathbb{P}_\mu(T_0>n-1),
$$
where $\mu$ denotes the distribution of $\xi_1$. \eqref{thm.log.3} follows now from
(2.4) in \cite{AM17}.
\end{proof}
We conclude this section with the following open problem: For which starting
distributions $\mu(dx)=\mathbb{P}(X_0\in dx)$ the statement of Theorem~\ref{thm.log}
remains valid? A general observation that the persistence exponent may depend
on the initial distribution of an AR(1)-sequence has been made in Proposition~2.2
in \cite{AM17}. The reader can find conditions on the initial distribution in  \cite{AM17}, in \cite{CV17} and also in Section 4.2 below ensuring the validity of Theorem 1. Until now a complete characterization does not seem to exist.  At this point we want to emphasize that this phenomenon is well known
in the theory of quasistationary distributions and is strongly related to the fact that
quasistationary distributions are not unique in general (see e.g. \cite{CMSM13}).
Conditions ensuring uniqueness of quasistationary distributions are given in \cite{CV16}.
%%%%%%%%%%%%%%%%%%%%%%%%%%%%%%%%%%%%%%%%%%%%%%%%%%%%%%%%%%%%%%%%%%%%%%%%%%%%%%%%%%%%%%%%%%%%%%%%%%%%%%%%
%%%%%%%%%%%%%%%%%%%%%%%%%%%%%%%%%%%%%%%%%%%%%%%%%%%%%%%%%%%%%%%%%%%%%%%%%%%%%%%%%%%%%%%%%%%%%%%%%%%%%%%%
\section{Innovations with bounded to the right support:\\ approach via quasistationarity}\label{s:boundedsupport}
Exponential decay of $\mathbb{P}(T_0>n)$ is often related to a quasi-stationary
behaviour of $(X_n)$ conditioned on the event $\{T_0>n\}$. Recall that the
quasistationarity implies that
$$
\mathbb{P}(T_0>n+1|T_0>n)\to e^{-\lambda_a}.
$$
So, the logarithmic asymptotics in Theorem~\ref{thm.log} should also follow
from the quasistionarity of $(X_n)$. Furthermore, it is quite natural to expect that
the  knowledge on the rate of convergence towards a quasi-stationary distribution
will imply preciser statements on the tail behaviour of $T_0$. This has been
recently confirmed in \cite{CV16}, where necessary and sufficient conditions for an exponentail speed of convergence in
the total variation norm to a quasistationary distribution have been provided. There, it has also been shown that such fast convergence yields
a purely exponential decay of $\mathbb{P}(T_0>n)$.

Let us formulate conditions from \cite{CV16} in terms of the AR(1)-sequence $(X_n)_{n\in\mathbb N_0}$.\\
{\bf First condition}: there exist a probability measure $\nu$ and constants
$n_0\ge1$, $c_1>0$ such that
\begin{equation}
\label{A1-cond}
\mathbb{P}_x(X_{n_0}\in\cdot|T_0>n_0)\ge c_1\nu(\cdot)\quad\text{for all }x>0.
\end{equation}
{\bf Second condition}: there exists a constant $c_2$ such that
\begin{equation}
\label{A2-cond}
\mathbb{P}_\nu(T_0>n)\ge c_2\mathbb{P}_x(T_0>n)\quad\text{for all }n\ge1\text{ and }x>0.
\end{equation}
It is rather obvious that \eqref{A2-cond} can not be valid for all $x>0$. This
observation implies that the results in \cite{CV16} are not
applicable to AR(1)-sequences with unbounded to the right innovations. So, we shall
assume that innovations $\xi_k$ are bounded. Let $R$ denote, as in the previous
section, the essential supremum of $\xi_1$. Then the invariant measure
lives on the set $(-\infty,R_*]$, where $R_*=R/(1-a)$. If the starting point lies
in $(0,R_*]$ then the chain $X_n$ does not exceed $R_*$ at all times. Consequently,
we have to find restrictions on the distribution of innovations which will ensure the 
validity of \eqref{A1-cond} and \eqref{A2-cond} for $x\le R_*$ only. 

We begin by showing that \eqref{A1-cond} holds for a quite wide class of innovations.
\begin{lemma}
\label{lem:A2-cond}
Assume that the distribution of innovations satisfy
\begin{equation}
\label{lem.A2.assump}
\mathbb{P}(\xi_1\le -aR_*)+\mathbb{P}(\xi_1=R)<1.
\end{equation}
Then there exist $\delta>0$  and a constant $c$ such that, for every $x\in[R_*-\delta,R_*]$,
\begin{equation}
\label{lem.A2.0}
\mathbb{P}_{R_*}(T_0>n)\le c\mathbb{P}_x(T_0>n),\quad n\ge1.
\end{equation}
In particular, the condition \eqref{A2-cond} is valid for any $\nu$ with 
$\nu[R_*-\delta,R_*]>0$.
\end{lemma}
\begin{remark}
If the assumption \eqref{lem.A2.assump} does not hold, i.e.,
$$
\mathbb{P}(\xi_1\le -aR_*)+\mathbb{P}(\xi_1=R)=1,
$$
then one can easily see that, for all $x\in(0,R_*]$ and $n\ge 1$,
$$
\mathbb{P}_x(T_0>n)=\left(\mathbb{P}(\xi_1=R)\right)^n.
$$
Therefore, \eqref{lem.A2.assump} does not restrict the generality.
\hfill
$\diamond$
\end{remark}
\begin{proof}[Proof of Lemma~\ref{lem:A2-cond}.]
Clearly, \eqref{lem.A2.assump} yields the existence of $\gamma>0$ such that
$$
\mathbb{P}(\xi_1>-aR_*+\gamma)>\mathbb{P}(\xi_1=R).
$$
Further, there exists $m=m(\gamma)$ such that, uniformly in starting points 
$x\in(0,R_*]$,
$$
\mathbb{P}_x(X_{m-1}>R_*-\gamma/a,T_0>m-1)\ge \left(\mathbb{P}(\xi_1=R)\right)^{m-1}.
$$
Consequently,
\begin{align*}
\mathbb{P}_x(T_0>m)
&\ge\mathbb{P}_x(X_{m-1}>R_*-\gamma/a,T_0>m-1)\mathbb{P}(\xi_m>-aR_*+\gamma)\\
&>\left(\mathbb{P}(\xi_1=R)+\varepsilon\right)^m
\end{align*}
for all $x\in(0,R_*]$ and some $\varepsilon>0$. Since
\begin{align*}
\mathbb P_x(T_0>nm) \ge \left( \min_{x>0}\mathbb P_x(T_0>m) \right)^{\left\lfloor\frac nm\right\rfloor} = \left( \mathbb P(\xi_1=R)+\varepsilon \right)^{\left\lfloor\frac nm\right\rfloor m}\,,
\end{align*}
we infer that
$$
\mathbb{P}(\xi_1=R)<e^{-\lambda_a}.
$$
Therefore, there exists $\delta>0$ such that
$$
\mathbb{P}(\xi_1>R-\delta)<e^{-\lambda_a},
$$
which is equivalent to
\begin{equation}
\label{lem.A2.1}
\varepsilon(\delta):=\mathbb{P}_{R_*}(X_1>R_*-\delta)<e^{-\lambda_a}.
\end{equation}

Taking into account the monotonicity property \eqref{eq.3}, we get
\begin{align*}
\mathbb{P}_{R_*}(T_0>n)
&\le \mathbb{P}_{R_*}(X_1\le R_*-\delta)\mathbb{P}_{R_*-\delta}(T_0>n-1)\\
&\hspace{1cm}+\mathbb{P}_{R_*}(X_1> R_*-\delta)\mathbb{P}_{R_*}(T_0>n-1)\\
&=(1-\varepsilon(\delta))\mathbb{P}_{R_*-\delta}(T_0>n-1)
+\varepsilon(\delta)\mathbb{P}_{R_*}(T_0>n-1).
\end{align*}
Iterating this estimate, we obtain
\begin{align}
\label{lem.A2.2}
\mathbb{P}_{R_*}(T_0>n)\le\frac{1-\varepsilon(\delta)}{\varepsilon(\delta)}
\sum_{k=1}^n\varepsilon^k(\delta)\mathbb{P}_{R_*-\delta}(T_0>n-k).
\end{align}
It follows from \eqref{eq.6a} that
$$
\mathbb{P}_{R_*-\delta}(T_0>n-k)\le
\frac{\mathbb{P}_{R_*-\delta}(T_0>n)}
{\mathbb{P}_{R_*-\delta}(T_0>k)\mathbb{P}_{R_*-\delta}(X_k>R_*-\delta)}.
$$
Plugging this into \eqref{lem.A2.2}, we have
\begin{align*}
\mathbb{P}_{R_*}(T_0>n)\le\frac{1-\varepsilon(\delta)}{\varepsilon(\delta)}
\frac{\mathbb{P}_{R_*-\delta}(T_0>n)}{\inf_k \mathbb{P}_{R_*-\delta}(X_k>R_*-\delta)} 
\sum_{k=1}^\infty\frac{\varepsilon^k(\delta)}{\mathbb{P}_{R_*-\delta}(T_0>k)}.
\end{align*}
The summability of the series on the right hand side follows from
Theorem~\ref{thm.log} and from estimate \eqref{lem.A2.1}. Furthermore,
the convergence of $X_n$ towards the stationary distribution $\pi$ implies
that $\inf_k \mathbb{P}_x(X_k>x)$ is positive. Thus, there exists a constant $c$
such that
$$
\mathbb{P}_{R_*}(T_0>n)\le c \mathbb{P}_{R_*-\delta}(T_0>n),\quad n\ge 1.
$$
The monotonicity of $\mathbb{P}_x(T_0>n)$ completes the proof of the first claim.

Using the monotonicity property once again and applying \eqref{lem.A2.0}, we obtain
\begin{align*}
\mathbb{P}_{\nu}(T_0>n)&\ge \frac{\nu[R_*-\delta,R_*]}{c}\mathbb{P}_{R_*}(T_0>n)\\
&\ge \frac{\nu[R_*-\delta,R_*]}{c}\mathbb{P}_{x}(T_0>n),\quad x\in(0,R_*].
\end{align*}
This completes the proof of the lemma.
\end{proof}

We now turn to the condition \eqref{A1-cond}.
\begin{lemma}
\label{lem:A1-cond}
Assume that the distribution of $\xi_1$ has an absolutely continuous component
with the density function $\varphi(x)$ satisfying
\begin{equation}
\label{lem.A1.1}
\varphi(y)\ge \varkappa>0\quad\text{for all }y\in[R-y_0,R],\ y_0>0. 
\end{equation}
Then, for every measurable $A\subseteq[R_*-y_0,R_*-ay_0]$,
$$
\liminf_{n\to\infty}\inf_{x\in[0,R_*]}\mathbb{P}_x(X_n\in A|T_0>n)\ge
\varkappa\pi[R_*-y_0,R_*){\rm Leb}(A)
$$
($\rm Leb$ denoting the Lebesgue measure).
\end{lemma}
\begin{proof}
By the Markov property,
\begin{align*}
\mathbb{P}_x(X_{n+1}\in A|T_0>n+1)
&=\frac{\mathbb{P}_x(X_{n+1}\in A,T_0>n+1)}{\mathbb{P}_x(T_0>n+1)}\\
&\ge\frac{\int_A\left(\int_0^{R_*}\varphi(z-ay)\mathbb{P}_x(X_n\in dy, T_0>n)\right)dz}{\mathbb{P}_x(T_0>n)}.
\end{align*}
Applying \eqref{lem.A1.1}, we get
\begin{align*}
\mathbb{P}_x(X_{n+1}\in A|T_0>n+1)
&\ge \varkappa\int_A\mathbb{P}_x\left(X_n\in\left[\frac{z-R}{a},\frac{z-R+y_0}{a}\right]\Big|T_0>n\right)dz.
\end{align*}
For every $z\ge R_*-y_0$ we have $\frac{z-R+y_0}{a}\ge R_*$. Therefore,
\begin{align*}
\mathbb{P}_x(X_{n+1}\in A|T_0>n+1)
&\ge \varkappa\int_A\mathbb{P}_x\left(X_n\ge\frac{z-R}{a}\Big|T_0>n\right)dz.
\end{align*}
Furthermore, for every $z\le R_*-ay_0$ one has $\frac{z-R}{a}\le R_*-y_0$.
Thus, using now \eqref{eq.5} and recalling that $X_n$ is increasing in the starting
point, we conclude that
\begin{align*}
\inf_{x\in[0,R_*)}\mathbb{P}_x(X_{n+1}\in A|T_0>n+1)
\ge \varkappa \mathbb{P}_0(X_n\ge R_*-y_0){\rm Leb}(A).
\end{align*}
Letting here $n\to\infty$, we get the desired estimate.
\end{proof}
Combining these two lemmata with Proposition 1.2 in \cite{CV16}, we get
\begin{theorem}
\label{prop.bounded}
Assume that the innovations $\xi_i$ are a.s. bounded and that their distribution possesses an absolutely continuous
component satisfying \eqref{lem.A1.1}. Then there exists a positive function
$V(x)$ such that, for each $x\in(0,R_*]$,
$$
\mathbb{P}_x(T_0>n)\sim V(x)e^{-\lambda_a n}\quad\text{as }n\to\infty.
$$
\end{theorem}
%%%%%%%%%%%%%%%%%%%%%%%%%%%%%%%%%%%%%%%%%%%%%%%%%%%%%%%%%%%%%%%%%%%%%%%%%%%%%%%%%%%%%%%%%%%%%%%%%%%%%%%%
%%%%%%%%%%%%%%%%%%%%%%%%%%%%%%%%%%%%%%%%%%%%%%%%%%%%%%%%%%%%%%%%%%%%%%%%%%%%%%%%%%%%%%%%%%%%%%%%%%%%%%%%
\section{Functional analytic approaches}\label{s:FuAn}
In this section we combine probabilistic insights with some basic functional analytic observations in
order to derive the precise tail behaviour of $T_0$. We want to stress that even though we call this
approach functional analytic we will only make use of rather fundamental properties of compact operators
combined with assertions of Perron-Frobenius type. The functional analytic ingredients can be found in
standard references such as \cite{AB06}, \cite{D07}, \cite{MN91} and \cite{S71}.

%\color{cyan}
\subsection{Quasi-compactness approach for bounded innovations.}\label{sec:quasi-bounded}
The initial idea from the introduction can be most straightforwardly carried through for bounded
innovations $\xi_i$. We assume that they have a density $\varphi$ which is strictly positive on all of
its support $[-A,B]$ ($A,B>0$) and consider $P_+f(x):=\mathbb Ef(ax+\xi_1)$ (where $f(y):=0$ for $y<0$)
as an operator on $C\left(\left[0,\frac B{1-a}\right]\right)$ with the supremum norm. In this case, we are going to show that
$P_+$ is compact with a simple largest eigenvalue $e^{-\lambda_a}$ strictly between 0 and 1 and can then
conclude
\begin{equation}
\label{expAbfb}
\mathbb P_x(T_0>n)=P_+^n\mathbf 1_{[0,\infty)}(x) = V(x)e^{-\lambda_an} + O(e^{-(\lambda_a+\varepsilon)n})
\end{equation}
for some nonnegative $V$ and $\varepsilon>0$. Apart from condition~\eqref{lem.A1.1}, which is not needed
in this approach, the result is contained in Theorem~\ref{prop.bounded}. Our main purpose here is to
finally lead over to those cases of unbounded innovations in which conditions from \cite{CV16} are not valid.

For any continuous $f$ and $x,y\in\left[0,\frac B{1-a}\right]$, $|P_+f(x)|\le\|f\|$ and
\begin{align}\label{gleichgr13}
|P_+f(x)-P_+f(y)| =& \Big|\int[\varphi(z-ax)-\varphi(z-ay)]f(z)\text dz\Big| \\
\le& \|f\|\int|\varphi(z-ax)-\varphi(z-ay)|\text dz\,.
\end{align}
This goes to zero for $y\to x$, i.e. $P_+$ maps bounded families to equicontinuous ones and, therefore,  is compact. $(X_n)$ clearly reaches zero from any starting point in $\left[0,\frac B{1-a}\right]$ if $\xi_1,\dots,\xi_{\left\lceil\frac{B}{1-a}\frac2A\right\rceil}<-\frac A2$, which happens with positive probability, so
\begin{align*}
\|P_+^{\left\lceil\frac{B}{1-a}\frac2A\right\rceil}\| = \|P_+^{\left\lceil\frac{B}{1-a}\frac2A\right\rceil}1_{[0,\infty)}\| = \sup_x\mathbb P_x\left(T_0>\left\lceil\frac{B}{1-a}\frac2A\right\rceil\right)<1\,.
\end{align*}
Consequently, all eigenvalues of $P_+$ must have modulus less than 1. We now invoke the following generalization of the Perron-Frobenius theorem (see Theorems 6 and 7 in \cite{S64}):\vspace{6pt}\\
\textbf{Theorem A} (see \cite{S64}).\textit{
Let $K$ be a proper closed cone in a Banach space $B$ which is fundamental and assume that $B$ is a lattice with respect to the ordering induced by $K$. Let $T:B\rightarrow B$ be quasicompact operator, which is positive with respect to $K$, i.e. $TK\subset K$. Further assume that for each $B\ni f>0, B^*\ni f^*>0$ there exists an integer $n(f,f^*)\geq 0$ such that $f^*(T^nf)>0$ for $n>n(f,f^*)$. Then 
\begin{itemize}
\item[a)] the spectral radius $r(T) \in \sigma(T)$ has algebraic multiplicity $1$ and is the only element in $\sigma(T)$ with absolute value equal to $r(T)$;
\item[b)] the eigenspace corresponding to the eigenvalue $r(T)$ is one-dimensional and is spanned by a strictly positive element $u$;
\item[c)] there exists a strictly positive element $u^*$ such that $T^*u^*=r(T) u^*$. 
\end{itemize} }

``Quasicompact'' should now be thought of as ``compact'', the general version will be needed and explained in the proof of Theorem~\ref{th:quasicompact}.

In our case, $B=C\left(\left[0,\frac B{1-a}\right]\right)$ and for $K$ we take the cone of nonnegative functions. It is closed, proper (i.e. $K\cap -K=\{0\}$) and fundamental (this means that $K$ spans a dense subset of $B$, which is clear since $K-K=B$). $P_+$, taking the role of $T$, is positive. The density $\varphi$ was assumed to be strictly positive, so positivity holds true even in the stronger sense that $P_+$ maps nonnegative functions which are somewhere strictly positive to functions which are everywhere strictly positive. For our choice of the space $B$ its dual space $B^*$ consists of all functionals $f\mapsto\int f\text d\mu$ with finite signed Borel measures $\mu$, positive functionals correspond to positive measures. Therefore, this strong positivity implies that $f^*(Tf)>0$ for all $f>0$ and all positive $f^*\in B^*$.

Theorem A is therefore applicable. It allows for the general conclusion that, for some $\varepsilon >0$ and any $f>0$,
\begin{equation}\label{Thochnallg}
T^n f = r(T)^n\,u^*(f)\,u + \mathcal{O}\bigl((r(T)-\varepsilon)^n\bigr)
\end{equation}
in the space $B$, which, applied to our setting with $f=\mathbf1_{\left[0,\frac{B}{1-a}\right]}$, yields \eqref{expAbfb}. 

In fact, the computation \eqref{gleichgr13} also works for bounded measurable $f$ and shows that $P_+$ is compact on the corresponding space $B\left(\left[0,\frac B{1-a}\right]\right)$. 
It should be noted that the corresponding dual space consists of signed finitely additive and absolutely continuous measures.
Consequently, \eqref{Thochnallg} is also applicable for indicator functions $f=\mathbf1_A$ with measurable $A\subset \left[0,\frac{B}{1-a}\right]$ and yields
$$\mathbb P_x(X_n\in A,T_0>n) = \nu(A)V(x)e^{-\lambda_an}+\mathcal O(e^{-(\lambda_a+\varepsilon)n})$$
with same factor $\nu(A)$ which is strictly positive unless $A$ has measure zero and, together with \eqref{expAbfb},
$$
\mathbb P_x(X_n\in A\mid T_0>n) = \nu(A) + \mathcal O(e^{-\varepsilon n}).
$$
In other words, we have an exponentially fast convergence to a quasistationary distribution $\nu$.

\subsection{Quasi-compactness approach for unbounded innovations}
If the innovations $\xi_i$ are absolutely continuous, but are unbounded, $P_+$ still maps bounded families to equicountinous ones and has the same positivity properties, but is in general not compact. A way out is to choose a more suitable Banach space. This can be, for example, explicitely carried through in the case of innovations with standard normal distributions. In this case, $X$  is a discretized Ornstein-Uhlenbeck process: The latter is given, e.g., by the stochastic differential equation
\begin{displaymath}
dZ_t = \theta Z_t\,dt+\sigma dB_t,\quad Z_0=x,
\end{displaymath}
where $\theta,\sigma>0$ are constants and $(B_t)_{t\geq 0}$ denotes a Brownian motion. The random variable $Z_t$ is normal distributed with 
\begin{displaymath}
\mathbb{E}[Z_t] = x e^{-\theta t}\quad \text{and}\quad \text{Var}[Z_t]=\frac{\sigma^2}{2\theta}(1-e^{-2\theta t}).
\end{displaymath}
For the chain $X_n$ with standard normal distributed innovations we have
\begin{displaymath}
\mathbb{P}_x\bigl(X_1 \in B) = \int_B\frac{1}{\sqrt{2\pi}}e^{-\frac{(ax-y)^2}{2}}\,dy.
\end{displaymath}
Taking 
\begin{displaymath}
\theta=\log(a^{-1})\quad \text{and}\quad \sigma^2=\frac{2\log(a^{-1})}{1-a^2}\,,
\end{displaymath}
we see that $Z_1$ and $X_1$ are identically distributed.

$Z$ has $\mathcal N(0,\frac{\sigma^2}{2\theta})$ as the stationary distribution and $Pf(x):=\mathbb Ef(ax+\xi_1)$ is a self-adjoint compact operator with norm 1 on $B:=L^2\left(\mathbb R,\mathcal N\left(0,\frac{\sigma^2}{2\theta}\right)\right)$. This is a well-known result and can be seen e.g. from its diagonal representation in the Hermite polynomials, which form a complete set of eigenfunctions. Consequently, also $$P_+:=1_{[0,\infty)}P1_{[0,\infty)}$$ is compact and self-adjoint. Its norm is strictly below 1. (Otherwise, one could find a normalised sequence $(f_n)$ with $\|P_+f_n\|\ge1-\frac1n$ for all $n\in\mathbb N$. A subsequence $(f_{n_k})$ would have a weak limit $f$ with $\|f\|\le1$ and, by compactness, $P_+f_{n_k}\to P_+f$ in norm, in particular $\|P_+f\|=1$. Then, $P1_{[0,\infty)}f$ would have norm 1, but the spectral representation of $P$ shows that $g\equiv1$ is the only function with $\|g\|\le1$ and $\|Pg\|=1$. Therefore, one can apply Theorem A in the same way as for bounded innovations. A result of a similar type has been shown in \cite{AB11}, but the conclusion drawn has been somewhat weaker.  

With somewhat more effort, one could also work on the space of continuous functions with the norm $\|f\|:=\sup_x|f(x)\sqrt{\pi(x)}|$, where $\pi$ is the density of the $\mathcal N(0,\frac{\sigma^2}{2\theta})$ distribution, so the crucial step was to introduce a suitable weight function, whereas $L^2$ instead of $C$ brought only computational benefits in this particular case.

Let us now consider general AR(1)-processes with possibly unbounded innovations.  Assume that for some $M>0$, some $\varepsilon >0$ and $\tilde{\lambda}=\lambda_a+\varepsilon$ we have
\begin{equation}\label{e:quasicomp}
\Lambda(x):=\mathbb{E}_x\bigl[ e^{\tilde{\lambda} {T}_M}\bigr] < \infty\,,
\end{equation}
where
$$
{T}_M:=\min\{k\ge1:X_k\le M\}\,.
$$
This property will be later shown to hold, whenever the innovations have moments of all orders, i.e. we do not require existence of an exponential moment. We introduce the Banach space $B(\mathbb R_0^+)$ of measurable functions on $[0,\infty)$ equipped with the norm 
\begin{displaymath}
\| f \|_\Lambda := \|\Lambda^{-1}f\|_{\infty} < \infty\,.
\end{displaymath}
$\Lambda^{-1}$ is decreasing in $x\in\mathbb R_0^+$ and clearly takes values in $(0,1]$. In contrast to other weight functions with these properties, it is computationally well tractable by the aid of the Markov property.

Actually the approach we now present is an adaption of a standard approach to the ergodicity of Markov chains via quasicompactness (see e.g. Chapter 6 in \cite{R84} as well as \cite{HH00}). In the literature on quasistationary distributions an approach of this type has been presented in \cite{G01}. Our main goal is to indicate via the example of autoregressive processes that the results of \cite{CV17} can be to a large extent reproved via the concept of quasicompact operators.  We hope that our outline of standard analytic ideas in a probabilistic context iluminates further the interrelationship between probabilistic and analytic concepts related to persistence exponents/quasistationarity on the one hand and spectral theoretic ideas on the other.

\begin{proposition}\label{prop:trop}
The transition operator $P$
\begin{displaymath}
Pf(x):=\mathbb{E}_x\bigl[f(X_1),{T}_0>1\bigr]
\end{displaymath}
defines a bounded operator on $(B(\mathbb{R}_0^+),\|\cdot\|_\Lambda)$. Moreover, the spectral radius $r(P)$ is lower bounded by $e^{-\lambda_a}$. 
\end{proposition}
\begin{proof}
For the boundedness, we observe that for $f \in B(\mathbb{R_+})$ with $\|f\|_\Lambda\leq 1$
\begin{equation*}
\begin{split}
\|Pf\|_\Lambda\leq \sup_{x >0}\bigl|\Lambda(x)^{-1}(P\Lambda)(x)\bigr| \,.
\end{split}
\end{equation*}
By the Markov property,
\begin{equation}\label{Markov7}
\mathbb E_x\left[e^{\tilde\lambda  T_M}\mid X_1\right]
= e^{\tilde\lambda}\mathbf1_{0<X_1\le M} + e^{\tilde\lambda} \Lambda(X_1)\mathbf1_{X_1>M} \,.
\end{equation}
Therefore, for every $x\ge0$ we have
\begin{equation*}
\begin{split}
(P\Lambda)(x)=& \mathbb E_x[\Lambda(X_1),0<X_1\le M] + \mathbb E_x[\Lambda(X_1),X_1>M] \\
\le&\Lambda(M) + e^{-\tilde\lambda}\mathbb E_x \left[e^{\tilde\lambda T_M},X_1>M\right]  \\
\le& \Lambda(M)+e^{-\tilde\lambda}\Lambda(x)\,,
\end{split}
\end{equation*}
proving that $P$ is bounded.
In order to prove the assertion concerning the spectral radius we observe that, for $x>0$,
\begin{displaymath}
\|P^n\|\geq \frac{1}{\Lambda(x)}(P^n\mathbf{1})(x)
\end{displaymath} 
and that therefore 
\begin{displaymath}
r(P):=\lim_{n\rightarrow \infty}\|P^n\|^{1/n}\geq e^{-\lambda_a}
\end{displaymath} 
by the Gelfand formula.
\end{proof}

\begin{theorem}\label{th:quasicompact}
Assume that condition \eqref{e:quasicomp} is satisfied and that the innovations are distributed according to a density $\varphi$. Then the operator $P$ is quasicompact on $(B(\mathbb{R}_0^+),\|\cdot\|_\Lambda)$. If, in addition, $\varphi>0$ a.e., the spectral radius $r(P)$ is an isolated eigenvalue with algebraic multiplicity $1$ and all other spectral values have absolut value strictly smaller than $r(P)$.
\end{theorem}
\begin{proof}
Let us decompose $P$ into the sum of the following operators: 
\begin{equation*}
\begin{split}
U_1:= P\mathbf{1}_{[0,M]}\,,\quad U_2:= \mathbf{1}_{[0,M]}P\mathbf{1}_{(M,\infty)}\,,\quad 
U_3:= \mathbf{1}_{(M,\infty)}P\mathbf{1}_{(M,\infty)}\,.
\end{split}
\end{equation*}
The operators $U_1$ and $U_2$ are easily seen to be compact. For this, we first note that 
\begin{displaymath}
\lbrace U_1f ; \|f\|_\Lambda\leq 1\rbrace  = \lbrace \mathbb{E}_x[f(X_1)\Lambda(X_1),0<X_1\leq M] ;\|f\|\leq 1\rbrace
\end{displaymath}
is bounded and equicontinuous. This follows immediately from the fact that $\Lambda(z)\le \Lambda(M)$ for $0\le z\le M$ and
\begin{align*}
\mathbb{E}_x[f(X_1)\Lambda(X_1),0<X_1\leq M] =  \int_0^M\Lambda(z)f(z)\varphi(z-ax)\text dz\,.
\end{align*} 
The boundedness also implies
$$
\lim_{x\to\infty}\sup_{f:\|f\|_\Lambda\le1}\Lambda^{-1}(x)\mathbb E_x[f(X_1),0<X_1\le M)]=0\,,
$$ 
so $\lbrace U_1f ; \|f\|_\Lambda\leq 1\rbrace$ is precompact and $U_1$ is compact.

By \eqref{Markov7}, for $x\le M$,
\begin{align*}
U_2(\Lambda f)(x) = \mathbb E_x\left[f(X_1)\Lambda(X_1),X_1>M\right]
= e^{-\tilde\lambda} \mathbb E_x\left[f(X_1)e^{\tilde\lambda T_M},X_1>M\right]\,,
\end{align*}
so for all $f$ with $\|f\|\le1$ we have 
$$|U_2(\Lambda f)(x)|\le e^{-\tilde\lambda}\Lambda(M)\,,$$
i.e. $\{U_2f ; \|f\|_\Lambda\le1\}$ is bounded w.r.t. $\|\cdot\|_\Lambda$, and
\begin{align*}
|U_2(\Lambda f)(y)-U_2(\Lambda f)(x)| \le e^{-\tilde\lambda}\sum_{k=2}^N |I_{f,k}(y)-I_{f,k}(x)| + 2e^{-\tilde\lambda}\sup_{0\le x\le M}|R_{1,N}(x)| 
\end{align*}
with
\begin{align*}
I_{f,k}(x) :=& e^{\tilde\lambda k}\mathbb E_x\left[f(X_1)\mathbf1_{X_1>M,\,T_M=k}\right] \\
=& e^{\tilde\lambda k}\int_M^\infty\dots\int_M^\infty\int_{-\infty}^Mf(x_1)\Phi_{x,k}(x_1,\dots,x_k)\text d(x_k,\dots,x_1)\,,
\end{align*}
where $$\Phi_{x,k}(x_1,\dots,x_k):=\varphi(x_1-ax)\varphi(x_2-ax_1)\dots\varphi(x_k-ax_{k-1})$$
stands for the common density of $X_1,\dots,X_k$ given $X_0=x$, and
\begin{align*}
R_{f,N}(x):= \mathbb E_x\left[f(X_1)e^{\tilde\lambda T_M},T_M>N\right]\,.
\end{align*}
Clearly, 
$$
\sup_{0\le x\le M}R_{1,N}(x)=R_{1,N}(M)\to0\quad\text{as }N\to\infty. 
$$
Moreover, 
$$
\Phi_{y,k}(x_1,\dots,x_{k})=\Phi_{x,k}(x_1-a(y-x),\dots,x_{k}-a^{k}(y-x))\,,
$$
so
\begin{align*}
|I_{f,k}(y)-I_{f,k}(x)|\le e^{\tilde\lambda k}\int\dots\int|\Phi_{y,k}-\Phi_{x,k}|\text d(x_1,\dots,x_k)\xrightarrow{y\to x}0\,.
\end{align*}
(This is clear for continuous $\Phi$ with compact support and follows easily for general $\Phi$ if one approximates them in $L^1$ by such ones.)
Consequently, $\{\left.(U_2f)\right|_{[0,M]}; \|f\|_\Lambda\le1\}$ is equicontinuous, $U_2$ is compact, too, and
$$
P=L+U_3,
$$
where $L$ is a compact operator.

We now estimate the operator norm of $U_3$: For $x>M$, \eqref{Markov7} yields
\begin{displaymath}
(U_3\Lambda)(x)=\mathbb{E}_x\bigl[\Lambda(X_1),X_1>M\bigr] \leq e^{-\tilde{\lambda}}\Lambda(x)
\end{displaymath}
and therefore we have for the operator norm of $U_3$
\begin{displaymath}
\|U_3\| \leq e^{-\tilde{\lambda}},
\end{displaymath}
which, again by the Gelfand formula, tells us that 
\begin{displaymath}
r(U_3)=\lim_{n\rightarrow\infty}\| U_3^{n}\|^{1/n} \leq e^{-\tilde{\lambda}}<e^{-\lambda_a}\le r(P)\,,
\end{displaymath}
the latter by Proposition~\ref{prop:trop}.

%Let us define the essential spectrum of $P$ to be
%\begin{displaymath}
%\sigma_{ess}(P):=\bigcap_{K\, \text{compact}}\sigma(P+K),
%\end{displaymath}
%where the intersection ranges through all compact operators, then it is known (see  e.g. Theorem 5.7 in \cite{S71}) that 
%\begin{displaymath}
%\lambda \notin \sigma_{ess}(P) \rightarrow \lambda-P^2 \quad \text{is Fredholm with index } 0.
%\end{displaymath}
%Observe that having Fredholm index equal $0$ means that either kernel or cokernel are  trivial or that both have equal dimension, in which case we get an eigenvalue as the kernel of $\lambda-P$ is non-trivial. Using that $U_3$ is a compact perturbation of $P$ and that $\|U_3\|<r(P)$ we conclude that eigenvalues with absolute value equal to $r(P)$ are isolated eigenvalues of finite algebraic and geometric multiplicities (see e.g. Theorem 4.3.18 in \cite{D07}. As a result, $P$ is quasicompact.
According to the definition in \cite{S64}, which calls an operator P quasicompact if $P^n = L + U$ for some $n\in\mathbb  N$, compact operator $L$ and  bounded operator $U$ with $\rho(U)<\rho(P)^n$, $P$ is quasicompact.
Applying Theorem A as in Section \ref{sec:quasi-bounded} allows to deduce the remaining statements of the theorem. 
\end{proof}
\begin{remark}
Observe that the main ingredient in the above proof is only the finitenes of $\mathbb{E}_x\bigl[e^{\tilde{\lambda}T_M}]$ for some $M>0$, which together with some 'local' compact perturbation argument ensures that there is a gap seperating the largest eigenvalue from the remaining parts of the spectrum. Therefore, we expect that the method will
work in other settings, too. 
\end{remark}
\begin{corollary}
\label{cor:qc}
Assume that condition \eqref{e:quasicomp} is satisfied and let us assume that the innovations have a strictly positive continuous density. Then there exists $\delta>0$ such that for every $x>0$ and every measurable set $A\subset (0,\infty)$
\begin{displaymath}
\mathbb{P}_x\bigl(X_n \in A;T_0>n)= V(x)e^{-\lambda_a n}+O\Bigl((e^{-\lambda_a}-\delta)^n\Bigr).
\end{displaymath}
\end{corollary}
This result is partly contained in the recent work \cite{CV17} by Champagnat and Villemonais.
If $\mathbb{E}e^{\theta(\xi_1)\log\xi_1}<\infty$ for some function $\theta(x)$ such that $\theta(x)\uparrow \infty$
and $\ell(x):=\theta(x)\log x$ is concave, then
$$
\frac{\mathbb{E}e^{\ell(ax+\xi_1)}}{e^{\ell(x)}}\le\frac{e^{\ell(ax)}\mathbb{E}e^{\ell(\xi_1)}}{e^{\ell(x)}}
\le e^{-\theta(ax)\log(a)}\mathbb{E}e^{\ell(\xi_1)}\to0\quad\text{as }x\to\infty
$$
and, consequently, in this case their Proposition 7.2 can be applied to AR$(1)$-sequences and gives 
the same type of convergence towards a quasi-stationary distribution. It is obvious that $\theta(x)=\log^b(x)$
with $b>0$ satisfies the conditions mentioned above. We shall see later that \eqref{e:quasicomp} holds for 
innovations having all power moments. Therefore, the condition $\mathbb{E}e^{\log^{1+b}\xi_1}<\infty$ is slightly
stronger than the existence of all power moments.

\section{Alternative approach for unbounded innovations}\label{s:renewal}
In this section we present another approach to investigate the tails of the hitting times, which, in contrast to Perron-Frobenius-type methods, has the potential to deal with situations where the transition operator is not quasicompact and to identify additional polynomial decay factors. Apart from birth/death processes and one-dimensional diffusions quasistationary convergence in cases with no spectral gap has not yet been established. Moreover, only basic properties of compact operators play a role and therefore the functional analytic machinery will be more straightforward.

As an alternative to the search for a weight function, we now start with the following observation: The larger $X_n$ is, the more it is diminished by the prefactor $a$ in the recursion $X_{n+1}=aX_n+\xi_{n+1}$. In contrast, adding $\xi_{n+1}$ has always the same absolute effect, no matter how large $X_n$ is. Therefore, in some sense, it is easier for the process to reach average positive values from extremely large values than to reach zero from average values, so, also in the case of unbounded innovations, the main part in estimating $\mathbb P_x(T_0>n)$ should still be to analyse what happens for not too large $x$.

The transition operator will in general not be (quasi-)compact on the usual space $C(\mathbb R_0^+)$, but, as shown on the next pages, a modified functional analytic approach, in which one basically works on the continuous functions on some interval and keeps under control what happens outside, is possible. In contrast, the ``coming down from infinity'' which we just described  has not the uniform character that would be needed to apply the results in \cite{CV16}.

Here is the main result of this section:
\begin{theorem}
\label{thm.analytic}
Assume that the distribution of innovations has a density $\varphi(x)$ which is positive a.e. on $\mathbb R$.
Assume also that $\mathbb{E}(\xi_1^+)^t<\infty$ for all $t>0$ and $\mathbb{E}(\xi_1^-)^\delta<\infty$ for some
$\delta>0$. Then there exist $\gamma>0$ 
%, $m\ge 0$, $0=\psi_0<\psi_1<\ldots,\psi_m<2\pi$ and functions
%$c_0(x), c_1(x),\ldots c_m(x)$ such that
and a positive function $V$ such that
\begin{equation}
\label{thm.an.1}
\mathbb{P}_x(T_0=n)=e^{-\lambda_a (n+1)}
%\sum_{j=0}^m c_j(x)\cos(\psi_j(n+1))
V(x)
+O\left(e^{-(\lambda_a+\gamma)n}\right).
\end{equation}
The function $V$ is $e^{\lambda_a}$-harmonic for the transition kernel $P_+$, that is,
$$
V(x)=e^{\lambda_a}\int_0^\infty P_+(x,dy)V(y)=e^{\lambda_a}\mathbb{E}[V(X_1);T_0>1],\quad x\ge0.
$$
\end{theorem} 

Again, this result describes not only the exact asymptotic behaviour of $\mathbb{P}_x(T_0=n)$ but states also
that the remainder term decays exponentially faster than the main term. It is
worth mentioning that the existence of all power moments required in Theorem~\ref{thm.analytic} is the 
minimal moment condition. More precisely, we shall show in Proposition~\ref{prop:reg.tails} that if the
tail of $\xi_1$ is regularly varying then it may happen that $e^{\lambda_a n}\mathbb{P}(T_0>n)\to0$.

%Summarising, the results from \cite{CV17} give a preciser information on the behaviour on the tail
%behaviour of $T_0$ under a slightly stronger moment restriction. But the

The starting point of the approach, which we are going to use in this
section, is based on the following renewal-type decomposition for the moment
generating function of $T_0$. 
First define
$$
\sigma_r:=\inf\{n\ge1: X_n> r\},\quad r>0.
$$
Fix $\lambda<\lambda_a$.
Then, for $x\le r$ we have
\begin{align*}
\mathbb{E}_x\left[e^{\lambda T_0}\right]
&=\mathbb{E}_x\left[e^{\lambda T_0}; T_0<\sigma_r\right]
+\mathbb{E}_x\left[e^{\lambda T_0}; T_0>\sigma_r\right]\\
&=\mathbb{E}_x\left[e^{\lambda T_0}; T_0<\sigma_r\right]
+\mathbb{E}_x\left[e^{\lambda \sigma_r}\mathbf{1}_{\{T_0>\sigma_r\}}
\mathbb{E}_{X_{\sigma_r}}\left[e^{\lambda T_0}\right]\right]\\
&=\mathbb{E}_x\left[e^{\lambda T_0}; T_0<\sigma_r\right]
+\mathbb{E}_x\left[e^{\lambda \sigma_r}\mathbf{1}_{\{T_0>\sigma_r\}}
\mathbb{E}_{X_{\sigma_r}}\left[e^{\lambda T_0};T_r=T_0\right]\right]\\
&\hspace{1cm}+\mathbb{E}_x\left[e^{\lambda \sigma_r}\mathbf{1}_{\{T_0>\sigma_r\}}
\mathbb{E}_{X_{\sigma_r}}\left[e^{\lambda T_0};T_r<T_0\right]\right].
\end{align*}
Using now the Markov property at time $T_r$, we obtain the equation
\begin{equation}
\label{main_decomp}
\mathbb{E}_x\left[e^{\lambda T_0}\right]=F_\lambda(x)+
\int_0^r K_\lambda(x,dy)\mathbb{E}_y\left[e^{\lambda T_0}\right],
\end{equation}
where
\begin{equation}
\label{def_F}
F_\lambda(x)=\mathbb{E}_x\left[e^{\lambda T_0}; T_0<\sigma_r\right]
+\mathbb{E}_x\left[e^{\lambda \sigma_r}\mathbf{1}_{\{T_0>\sigma_r\}}
\mathbb{E}_{X_{\sigma_r}}\left[e^{\lambda T_0};T_r=T_0\right]\right]
\end{equation}
and
\begin{equation}
\label{def_K}
K_\lambda(x,dy)=
\mathbb{E}_x\left[e^{\lambda \sigma_r}\mathbf{1}_{\{T_0>\sigma_r\}}
\mathbb{E}_{X_{\sigma_r}}\left[e^{\lambda T_r}\mathbf{1}_{\{T_r<T_0\}}\mathbf 1_{\text dy}(X_{T_r})\right]\right]\,.
\end{equation}

To analyse the renewal equation \eqref{main_decomp}, we first have to derive some
properties of the functions $F_\lambda$ and the operators $K_\lambda$. More precisely,
we first show that there exists $r>0$ such that $F_\lambda(x)$ and $K_\lambda(x,dy)$
can be extended analytically for $\Re\lambda<\lambda_a+\varepsilon$. 

\subsection{Estimates for stopping times $T_0\wedge\sigma_r$ and $T_r$.}

The main purpose of this paragraph is to show that, under the conditions of
Theorem~\ref{thm.analytic}, $T_0\wedge\sigma_r$ and $T_r$ have lighter tails than $T_0$. 
This fact will play a crucial role in the study of properties of $F_\lambda$ and 
$K_\lambda$.

\begin{lemma}
\label{lem:two-sided}
Assume that $\mathbb{E}|\xi_1|^\delta$ is finite for some $\delta>0$. Then
for every $r$ such that $\pi[r,\infty)>0$ there exists $\varepsilon_r>0$ such that
\begin{equation}
\label{two-sided.1}
\sup_{x\in(0,r)}\mathbb{P}_x(T_0\wedge\sigma_r>n)\le C_re^{-(\lambda_a+\varepsilon_r)n},
\quad n\ge0.
\end{equation}
\end{lemma}
\begin{proof}
Clearly,
$$
\mathbb{P}_x(T_0\wedge\sigma_r>n)=\mathbb{P}_x\left(\max_{k\le n}X_k<r, T_0>n\right),\quad n\ge1.
$$
Fix some $n_0\ge1$ and consider the sequence
$\mathbb{P}_x\left(\max_{k\le \ell n_0}X_k<r, T_0>\ell n_0\right)$ in $\ell\in\mathbb N$. By the Markov property,
\begin{align*}
&\mathbb{P}_x\left(\max_{k\le \ell n_0}X_k<r, T_0>\ell n_0\right)\\
&=\int_0^r\mathbb{P}_x\left(X_{(\ell-1)n_0}\in dy;\max_{k\le (\ell-1) n_0}X_k<r, T_0>(\ell-1) n_0\right) 
\mathbb{P}_y\left(\max_{k\le n_0}X_{k}<r,T_0>n_0\right).
\end{align*}
It is easy to see that functions $\mathbf{1}_ {\{\max_{k\le n_0}X_k\ge r\}}$ and
$\mathbf{1}_{\{T_0>n_0\}}$ are increasing functions in every innovation $\xi_k$,
$k\le n_0$. Thus, by the FKG-inequality for product spaces,
$$
\mathbb{P}_y\left(\max_{k\le n_0}X_{k}\ge r,T_0>n_0\right)\ge
\mathbb{P}_y\left(\max_{k\le n_0}X_{k}\ge r\right)\mathbb{P}_y\left(T_0>n_0\right).
$$
In other words,
\begin{align*}
\mathbb{P}_y\left(\max_{k\le n_0}X_{k}<r,T_0>n_0\right)
&\le \mathbb{P}_y\left(\max_{k\le n_0}X_{k}<r\right)\mathbb{P}_y\left(T_0>n_0\right)\\
&\le \mathbb{P}_0\left(\max_{k\le n_0}X_{k}<r\right)\mathbb{P}_r\left(T_0>n_0\right).
\end{align*}
Consequently,
\begin{align}
\label{two-sides.3}
\nonumber
&\mathbb{P}_x\left(\max_{k\le \ell n_0}X_k<r, T_0>\ell n_0\right)\\
\nonumber
&\hspace{0.5cm}\le\mathbb{P}_0\left(\max_{k\le n_0}X_k<r\right)\mathbb{P}_r(T_0>n_0)
\mathbb{P}_x\left(\max_{k\le (\ell-1) n_0}X_k<r, T_0>(\ell-1) n_0\right)\\
&\hspace{0.5cm}\le\ldots\le
\left(\mathbb{P}_0\left(\max_{k\le n_0}X_k<r\right)\mathbb{P}_r(T_0>n_0)\right)^\ell,
\quad x\in(0,r).
\end{align}
By Theorem~\ref{thm.log}, $\mathbb{P}_r(T_0>n_0)=e^{-\lambda_a n_0+o(n_0)}$ as $n_0\to\infty$. Since
$r-X_n$ is an AR($1$)-sequence, we may apply Theorem~\ref{thm.log} to this sequence:
$$
\mathbb{P}_0\left(\max_{k\le n_0}X_k<r\right)=e^{-\widetilde{\lambda}_an_0+o(n_0)},
\quad n_0\to\infty
$$
for some $\widetilde{\lambda}_a>0$. Therefore, there exists $n_0$ such that
$$
\frac{1}{n_0}\log\mathbb{P}_0\left(\max_{k\le n_0}X_k<r\right)\mathbb{P}_r(T_0>n_0)<-\lambda_a-\frac{\widetilde{\lambda}_a}2\,.
$$
Combining this estimate with \eqref{two-sides.3}, we obtain \eqref{two-sided.1}. 
\end{proof}

We next show that a similar estimate holds for $T_r$.
\begin{lemma}
\label{prop:all.mom}
Assume that $\mathbb{E}(\xi_1^+)^t<\infty$ for all $t>0$. Then for all $A\in(1,1/a)$
and all $\lambda>0$
there exists $r_0=r_0(A,\lambda)$ such that, for all $r\ge r_0$,
$$
\mathbb{E}\bigl[ e^{\lambda T_r}\bigr]\le 2e^\lambda\mathbb{E}\left(\frac{X_0}{r}\right)^{\lambda/\log A}
$$
for any distribution of $X_0$ with support in $(r,\infty)$.
\end{lemma}
\begin{proof}
Define
$$
y_j=rA^j,\quad j\ge0.
$$
Let us consider an 'aggregated' chain $(Y_n)_{n\ge 0}$ defined by the transition kernel
$$
\mathbb{P}(Y_1=j|Y_0=k)=\mathbb{P}_{y_k}(X_1\in(y_{j-1},y_j]),\ k,j\ge1 
$$
and
$$
\mathbb{P}(Y_1=0|Y_0=0)=1,\quad
\mathbb{P}(Y_1=0|Y_0=k)=\mathbb{P}_{y_k}(X_1\le y_0)\,,\, k\ge1.
$$
Similarly we define the initial distribution:
$$
\mathbb{P}(Y_0=j)=\mathbb{P}(X_0\in(y_{j-1},y_j])\,,\,j\ge1\,.
$$
Define also the stopping times
$$
\tau:=\inf\{n: Y_n=0\}
\quad\text{and}\quad
\theta:=\inf\{n:Y_n\ge Y_{n-1}\}.
$$
We first estimate the distribution of $Y_\theta$. Noting that $Y_\theta=j\ge1$ implies that
$Y_{\theta-1}\le j$, we have 
$$
\max_{k\ge1}\mathbb{P}(Y_\theta=j|Y_0=k)=\max_{k\le j}\mathbb{P}(Y_\theta=j|Y_0=k).
$$
Furthermore, using the fact that $\theta\le j$ for $Y_0\le j$ and $Y_\theta=j$, we
infer that
$$
\max_{k\ge1}\mathbb{P}(Y_\theta=j|Y_0=k)\le j\max_{k\le j}\mathbb{P}(Y_1=j|Y_0=k).
$$
By the definition of $Y_n$,
\begin{align*}
\mathbb{P}(Y_1=j|Y_0=k)
&=\mathbb{P}_{y_k}(X_1\in(y_{j-1},y_j])\\
&\le \mathbb{P}_{y_k}(X_1>y_{j-1})\le\mathbb{P}(\xi_1>(1-aA)u_{j-1}),\quad k\le j.
\end{align*}
As a result,
\begin{equation}
\label{all.mom.1}
\max_{k\ge1}\mathbb{P}(Y_\theta=j|Y_0=k)\le j\mathbb{P}(\xi_1>(1-aA)u_{j-1})=:q_j(r),
\quad j\ge1. 
\end{equation}
Set $q_0(r):=1-\sum_{j=1}^\infty q_j(r)$.
It is easy to see that $\lim_{r\to\infty}\sum_{j=1}^\infty q_j(r)=0$.
Therefore, $\{q_j(r)\}$ is a probability distribution for all $r$ large enough.

Noting that the chain $(Y_n)$ is stochastically monotone  and using \eqref{all.mom.1}, 
we have, for arbitrary initial $Y_0$,
$$
\tau_0\le Y_0+\mathbf{1}_{\lbrace Y_\theta \ge 1\rbrace}\tau^{(q)}_0\quad\text{in distribution},
$$
where $\tau^{(q)}_0$ is independent of $Y_0,\ldots,Y_\theta$ and is distributed
as $\tau_0$ corresponding to the initial distribution $q_j(r)/(1-q_0(r)),\ j\ge1$.
Combining  this inequality with the bound $\mathbb{P}(Y_\theta\ge1)\le 1-q_0(r)$,
we obtain
\begin{equation}
\label{all.mom.2}
\mathbb{E} e^{\lambda\tau_0}
\le\mathbb{E}e^{\lambda Y_0}\left(1+(1-q_0(r))\mathbb{E}e^{\lambda\tau^{(q)}_0}\right).
\end{equation}

If $Y_0$ is distributed according to $q_j(r)/(1-q_0(r))$, then we conclude from \eqref{all.mom.2} that
\begin{equation}
\label{all.mom.3}
\mathbb{E}e^{\lambda\tau^{(q)}_0}
\le \frac{\mathbb{E}e^{\lambda Y_0}}{1-(1-q_0(r))\mathbb{E}e^{\lambda Y_0}}.
\end{equation}
It follows from the Chebyshev inequality that
$$
q_j(r)\le j\frac{\mathbb{E}(\xi_1^+)^t}{(1-aA)^tu_{j-1}^t}
=\frac{\mathbb{E}(\xi_1^+)^tA^t}{(1-aA)^tr^t}jA^{-jt},\quad j\ge1.
$$
Consequently,
\begin{align*}
(1-q_0(r))\mathbb{E}e^{\lambda Y_0}&=\sum_{j=1}^\infty q_j(r)e^{\lambda j}\\
&\le \frac{\mathbb{E}(\xi_1^+)^te^\lambda}{(1-aA)^tr^t}\sum_{j=1}^\infty j(e^\lambda A^{-t})^{j-1}
=\frac{\mathbb{E}(\xi_1^+)^te^\lambda}{(1-aA)^tr^t}\left(1-e^\lambda A^{-t}\right)^{-2}.
\end{align*}
For all $t$ such that $A^t\ge 2e^{\lambda}$ we have
$$
(1-q_0(r))\mathbb{E}e^{\lambda Y_0}\le 4\frac{\mathbb{E}(\xi_1^+)^te^\lambda}{(1-aA)^tr^t}.
$$
As a result,
$$
(1-q_0(r))\mathbb{E}e^{\lambda Y_0}\le\frac{1}{2}
$$
for all $r$ large enough. Combining this estimate with \eqref{all.mom.3}, we obtain
$$
(1-q_0(r))\mathbb{E}e^{\lambda \tau^{(q)}_0}\le 1.
$$
Plugging this into \eqref{all.mom.2} leads to
$$
\mathbb{E} e^{\lambda\tau_0}\le2\mathbb{E}e^{\lambda Y_0}.
$$
The stocahstic monotonicity of $X_n$ implies that $T_r\le\tau_0$ in distribution. Combining this with
the fact that $Y_0\le \frac{\log(X_0/r)}{\log A}+1$, we obtain the desired inequality for the chain $X_n$.
\end{proof}
\begin{corollary}
\label{cor:all.mom}
Under the assumptions of Lemma~\ref{prop:all.mom}, there exist $r>0$ and $\varepsilon>0$ such that, for all $\lambda\le\lambda_a+2\varepsilon$,
$$
\sup_{x\in[0,r]}
\mathbb{E}_x\left[e^{\lambda\sigma_r}\mathbf{1}\{T_0>\sigma_r\}\mathbb{E}_{X_{\sigma_r}}[e^{\lambda T_r}]\right]
<\infty\,.
$$
\end{corollary}
\begin{proof}
By the Markov property at time $\sigma_r$ and by Lemma~\ref{prop:all.mom},
\begin{align}
\label{revis.1}
\nonumber
&\mathbb{E}_x\left[e^{\lambda\sigma_r}\mathbf{1}\{T_0>\sigma_r\}\mathbb{E}_{X_{\sigma_r}}[e^{\lambda T_r}]\right]\\
\nonumber
&\hspace{1cm}=\sum_{k=1}^\infty e^{\lambda k}\int_r^\infty\mathbb{P}_x(T_o>\sigma_r=k,X_{\sigma_r}\in dy)\mathbb{E}_y[e^{\lambda T_r}]\\
&\hspace{1cm}\le\frac{2e^\lambda}{r^{\lambda/\log A}}\sum_{k=1}^\infty e^{\lambda k}\int_r^\infty\mathbb{P}_x(T_o>\sigma_r=k,X_{\sigma_r}\in dy)y^{\lambda/\log A}.
\end{align}
Noting that 
\begin{align}
\label{revis.2}
\mathbb{P}_x(T_0>\sigma_r=k,X_{\sigma_r}>z)\le \mathbb{P}_x(T_0>\sigma_r>k-1)\mathbb{P}(\xi_1>(1-a)z),\quad z>r,
\end{align}
we obtain
\begin{align*}
\mathbb{E}_x\left[e^{\lambda\sigma_r}\mathbf{1}\{T_0>\sigma_r\}\mathbb{E}_{X_{\sigma_r}}[e^{\lambda T_r}]\right]
\le C(\lambda,A,r)\sum_{k=1}^\infty e^{\lambda k} \mathbb{P}_x(T_0>\sigma_r>k-1).
\end{align*}
The desired bound follows now from Lemma~\ref{lem:two-sided}.
\end{proof}
\subsection{Properties of $F_\lambda$ and $K_\lambda$.}
We start this subsection by stating  properties of the function $F_\lambda$ that are important for our approach.
\begin{lemma}\label{l:pos-function}
For every complex $z$ with real part $\Re (z)\le \lambda_a+\varepsilon$ the function $F_{z}$ defines a continuous
function on $[0,r]$, which is strictly positive if additionally $z\in \mathbb{R}$. 
\end{lemma}
\begin{proof}
We first show that $F_z$ is well-defined for all $z$ with $\Re (z)<\lambda_a+\varepsilon$. Indeed, it follows
from Lemma~\ref{lem:two-sided} that
\begin{equation*}
\Big|\mathbb{E}_x\left[e^{z T_0};T_0<\sigma_r\right]\Big|
\le \mathbb{E}_x\left[e^{\Re (z) T_0\wedge\sigma_r}\right]\le C 
\end{equation*}
uniformly in $x\in[0,r]$ and in $z$ with $\Re (z)<\lambda_a+\varepsilon$ for every $\varepsilon<\varepsilon_r$.
Applying  Corollary~\ref{cor:all.mom}, we also conclude that
\begin{align*}
&\Big|\mathbb{E}_x\left[e^{z \sigma_r}\mathbf{1}_{\{T_0>\sigma_r\}}
\mathbb{E}_{X_{\sigma_r}}\left[e^{z T_0};T_r=T_0\right]\right]\Big|\\
&\hspace{2cm}\le \mathbb{E}_x\left[e^{\Re (z) \sigma_r}\mathbf{1}_{\{T_0>\sigma_r\}}
\mathbb{E}_{X_{\sigma_r}}\left[e^{\Re (z) T_r}\right]\right]\le C
\end{align*}
uniformly in $x\in[0,r]$ and in $z$ with $\Re (z)<\lambda_a+\varepsilon$. Therefore, $F_z(x)$ is bounded
for all $x\in[0,r]$ and all $z$ with $\Re (z)<\lambda_a+\varepsilon$.

It is also clear that, for all $x\in[0,r]$,
$$
F_\lambda(x)\ge \mathbb{E}_x\left[ e^{\lambda(T_0\wedge\sigma_r)}\right]\ge 1,
\quad \lambda\in[0,\lambda_a+\varepsilon).
$$
Thus, it remains to show the continuity of $F_z$. Fix some $N\ge1$. Since the innovations
have an absolutely continuous distribution, the probabilities $\mathbb{P}_x(\sigma_r>T_0=k)$
are continuous in $x$. As a result, $\sum_{k=1}^N e^{zk}\mathbb{P}_x(\sigma_r>T_0=k)$ is continuous in $x$.
Noting that, according to Lemma~\ref{lem:two-sided},
\begin{align*}
\max_{x\in[0,r]}\Big|\mathbb{E}_x\left[e^{z T_0}\mathbf{1}_{\{\sigma_r>T_0>N\}}\right]\Big|\to0
\quad\text{as }N\to\infty,
\end{align*}
we infer that $\mathbb{E}_x\left[e^{z T_0};T_0<\sigma_r\right]$ is continuous on $[0,r]$.

Using the continuity of the distribution of innovations once again, we conclude that
$$
\mathbb{E}_x\left[e^{z\sigma_r}\mathbf{1}_{\{\sigma_r<T_0\wedge N\}}\mathbb{E}_{X_{\sigma_r}}[e^{z T_r}\mathbf{1}_{\{T_r=T_0<N\}}]\right]
$$
is continuous in $x\in[0,r]$. Therefore, it remains to show that, as $N\to\infty$,
\begin{equation}
\label{pos.f.1}
\sup_{x\in[0,r]}\Big|\mathbb{E}_x\left[e^{z\sigma_r}\mathbf{1}_{\{T_0>\sigma_r\ge N\}}\mathbb{E}_{X_{\sigma_r}}[e^{z T_r}]\right]\Big|\to0
\end{equation}
and 
\begin{equation}
\label{pos.f.2}
\Big|\mathbb{E}_x\left[e^{z\sigma_r}\mathbb{E}_{X_{\sigma_r}}[e^{z T_r}\mathbf{1}_{\{T_r\ge N\}}]\right]\Big|\to0.
\end{equation}

Similar to the derivation of \eqref{revis.1},
\begin{align*}
&\left|\mathbb{E}_x\left[e^{z\sigma_r}\mathbf{1}_{\{T_0>\sigma_r\ge N\}}\mathbb{E}_{X_{\sigma_r}}[e^{z T_r}]\right]\right|\\
&\hspace{1cm} \le \mathbb{E}_x\left[e^{\Re(z)\sigma_r}\mathbf{1}_{\{T_0>\sigma_r\ge N\}}\mathbb{E}_{X_{\sigma_r}}[e^{\Re(z) T_r}]\right]\\
&\hspace{1cm} \le \frac{2e^{\Re(z)}}{r^{\Re(z)/\log A}}\sum_{k=N}^\infty e^{\Re(z) k}\int_r^\infty\mathbb{P}_x(T_o>\sigma_r=k,X_{\sigma_r}\in dy)y^{\Re(z)/\log A}.
\end{align*}
Using now \eqref{revis.2}, one gets
$$
\left|\mathbb{E}_x\left[e^{z\sigma_r}\mathbf{1}_{\{T_0>\sigma_r\ge N\}}\mathbb{E}_{X_{\sigma_r}}[e^{z T_r}]\right]\right|
\le C(\Re(z),A,r)\sum_{k=N}^\infty e^{\Re(z) k}\mathbb{P}(T_0>\sigma_r>k-1).
$$
From this bound and Lemma~\ref{lem:two-sided} we infer that \eqref{pos.f.1} is valid for all $z$ with $\Re (z)<\lambda_a+2\varepsilon$.

Applying the Cauchy-Schwarz inequality, we obtain
$$
\left|\mathbb{E}_y\left[e^{zT_r}\mathbf{1}_{\{T_r\ge N\}}\right]\right|
\le \left(\mathbb{E}_y\left[e^{2\Re(z)T_r}\right]\right)^{1/2}\mathbb{P}_y^{1/2}(T_r\ge N).
$$
From this bound and Lemma~\ref{prop:all.mom}, we get 
\begin{align*}
\Big|\mathbb{E}_x\left[e^{z\sigma_r}\mathbb{E}_{X_{\sigma_r}}[e^{z T_r}\mathbf{1}_{\{T_r\ge N\}}]\right]\Big|
\le C\mathbb{E}_x\left[e^{\Re(z)\sigma_r}X_{\sigma_r}^{\Re(z)/\log A}\mathbb{P}^{1/2}_{X_{\sigma_r}}(T_r\ge N)\right]
\end{align*}
Since $\lim_{N\to\infty}\mathbb{P}^{1/2}_{X_{\sigma_r}}(T_r\ge N)$ almost surely and the family
$e^{\Re(z)\sigma_r}X_{\sigma_r}^{\Re(z)/\log A}$, $\Re(z)\le \lambda_a+\varepsilon$ is uniformly integrable,
we conclude that \eqref{pos.f.2} is valid. 
\end{proof}

We now establish an essential further property of the family of kernels $K_{\lambda}$, which exactly is the reason for introducing them. 
\begin{lemma}
\label{lem:K_lambda}
Consider $K_\lambda f:=\int_0^rK_\lambda(\cdot,dy)f(y)$, where $r$ is as in Lemma~\ref{prop:all.mom}.
Then, under the assumptions of Theorem~\ref{thm.analytic}, for all $\lambda$ with
$\Re(\lambda)\le\lambda_a+\varepsilon$, $K_\lambda$ is a compact operator on the Banach space $X=C([0,r])$ equipped with the supremum norm.

Furthermore, if $\lambda\in[0,\lambda_a+\varepsilon)$ then for every $X \ni f > 0$ (i.e. everywhere nonnegative and somewhere strictly positive)
and  all $x\in[0,r]$, $K_{\lambda}f(x)>0$.
\end{lemma}

\begin{proof}
We have to show that, for fixed $\lambda$, $\{K_\lambda f: f\in C([0,r]),\|f\|\le1\}$ is equicontinuous. This will imply that $K_\lambda$ maps to $C([0,r])$, in particular $\|K_\lambda 1\|<\infty$, and the proof can then be concluded by noting that $\sup\{\|K_\lambda f\| :f\in C([0,r]),\|f\|\le1\}=\|K_\lambda1\|<\infty$, which, together with equicontinuity, yields compactness.

For equicontinuity, one has to bring the smoothing effect of the density into play. We write
$$\Phi_{x,k,l}(x_1,\dots,x_{k+l}):=\varphi(x_1-ax)\cdots\varphi(x_{k+l}-ax_{k+l-1})$$
for the common density of $X_1,\dots,X_{k+l}$ and write
\begin{align*}
&K_\lambda f(x) = \sum_{k,l=1}^NI_{f,k,l}(x) + R_{f,N,1}(x)+R_{f,N,2}(x)
\end{align*}
and
\begin{equation}\label{comp1}
\begin{split}
&|K_\lambda f(y)-K_\lambda f(x)|  \\
\le& \sum_{k,l=1}^N|I_{f,k,l}(y)-I_{f,k,l}(x)| + 2\sup_{x\in[0,r]}\left(|R_{1,N,1}(x)|+ |R_{1,N,2}(x)|\right)
\end{split}
\end{equation}
with
\begin{align*}
I_{f,k,l}(x) =&   e^{\lambda(k+l)}\mathbb{E}_x\bigl[\mathbf{1}_{X_1.\dots,X_{k-1}\in]0,r[,\, X_k\ge r,\, X_{k+1},\dots,X_{k+l-1}>r,\, X_{k+l}\in]0,r]} f(X_{k+l})\bigr] \\
=& e^{\lambda(k+l)}\int_0^r\dots\int_0^r\int_r^\infty\dots\int_r^\infty\int_0^r\Phi_{x,k,l}(x_1,\dots,x_{k+l})f(x_{k+l})\text dx_{k+l}
%\text dx_{k+l-1}\cdots\text dx_{k}\text dx_{k-1}
\cdots\text dx_1\,,
\end{align*}
\begin{align*}
R_{f,N,1}(x)=\mathbb{E}_x\bigl[e^{\lambda\sigma_r}\mathbf{1}_{N<\sigma_r<T_0}\mathbb{E}_{X_{\sigma_r}}\bigl[e^{\lambda T_r}\mathbf{1}_{T_r<T_0}f(X_{T_r})\bigr]\bigr]
\end{align*}
and
\begin{align*}
R_{f,N,2}(x) = \mathbb{E}_x\bigl[e^{\lambda\sigma_r}\mathbf{1}_{\lbrace \sigma_r<T_0, \sigma_r\le N\rbrace}\mathbb{E}_{X_{\sigma_r}}\bigl[e^{\lambda T_r}\mathbf{1}_{\lbrace N<T_r<T_0\rbrace}f(X_{T_r})\bigr]\bigr] \,.
\end{align*}
Since
$$\Phi_{y,k,l}(x_1,\dots,x_{k+l})=\Phi_{x,k,l}(x_1-a(y-x),\dots,x_{k+l}-a^{k+l}(y-x))\,,$$
we can conclude
\begin{align*}
|I_{f,k,l}(y)-I_{f,k,l}(x)|
\le \int_0^\infty\dots\int_0^\infty|\Phi_{y,k,l}-\Phi_{x,k,l}|\text d(x_1,\dots,x_{k+l})\xrightarrow{y\to x}0\,,
\end{align*}
independently of $f$. (This is clear for continuous $\Phi$ with compact support and follows easily for general $\Phi$ if one approximates them in $L^1$ by such ones.) 

It follows from \eqref{pos.f.1} and \eqref{pos.f.2} that
\begin{align*}
\lim_{N\to\infty}\sup_{x\in[0,r]}|R_{1,N,1}(x)|=0
\end{align*}
and
\begin{align*}
\lim_{N\to\infty}\sup_{x\in[0,r]}|R_{1,N,2}(x)|=0.
\end{align*}
So, if we go back to equation~\eqref{comp1}, choose first $N$ large and then $y$ sufficiently close to $x$, equicontinuity follows.

We now prove the positivity of $K_\lambda$ for real values of $\lambda$. 
If $f(z)\ge\varepsilon>0$ for $z\in I:=[z_0-\delta,z_0+\delta]$, then
\begin{equation*}
K_\lambda f(x)\ge \varepsilon \mathbb{E}_x\left[\mathbf{1}_{\{T_0>\sigma_r\}}
\mathbb P_{X_{\sigma_r}}\left(T_r<T_0,X_{T_r}\in I\right)\right]\,.
\end{equation*}
The expression is bounded from below by $\varepsilon\mathbb P_x(X_1\in(r,r+\delta],X_2\in I)$, which is positive if the density $\varphi$ is positive on all of $\mathbb R$.

\end{proof}
\subsection{Fredholm alternative and the proof of Theorem~\ref{thm.analytic}}
We will now make use of the so-called analytic Fredholm alternative. For the convenience of the reader we formulate a suitable version of this result.
\vspace{6pt}\\
\textbf{Theorem B} (Theorem 1 in \cite{S68}).\textit{
Let $(X,\|\cdot\|_X)$ be a complex Banach space and let $\mathcal{B}(X)$ denote the space of bounded linear endomorphisms. Let $D$ be an open connected subset of $\mathbb{C}$. Let $A:D \mapsto \mathcal{B}(X)$ be an operator-valued analytic function such that $A(z)$ is a compact operator for every $z\in D$. Then either:
\begin{enumerate}
\item $(I-A(z))^{-1}$ does not exist for any $z \in D$, or
\item $(I-A(z))^{-1}$ exists for all $z \in D\setminus S$, where $S\subset D$ is discrete. 
\end{enumerate}}

We are now going to apply this result in order to deduce the subsequent proposition:
\begin{proposition}
For $z \in \mathbb{C}$ with $\Re z <\lambda_a + \varepsilon$, the operator valued function 
\begin{displaymath}
\lbrace z \in \mathbb{C} \mid \Re z <\lambda_a + \varepsilon\rbrace \ni z\mapsto R_z:= (I-K_z)^{-1}
\end{displaymath}
is meromorphic.
\end{proposition}
\begin{proof}
It suffices to show that for some $z \in \mathbb{C}$ with $\Re z <\lambda_a+\varepsilon$  the inverse of $I-K_z$ exists and defines a bounded operator. This follows from the fact that for $\lambda$ sufficiently small the operator norm of $K_{\lambda}$ can be easily seen to be strictly smaller than one and the inverse thus can be shown to exist by the Neumann series.  An application of Theorem B therefore shows that $(I-K_z)^{-1}$ is meromorphic on the required domain. 
\end{proof}
Now we are able to draw a conclusion analogously to Corollary 1 of \cite{SW75}:
\begin{corollary}
If for  $z \in \mathbb{C}$ with $\Re z <\lambda_a + \varepsilon$ the function $u=u_z$ is a solution of 
\begin{displaymath}
u=F_z+K_zu
\end{displaymath}
then the $X$-valued function $u$ is meromorphic in $z$. Therfore, we conclude that the function 
\begin{displaymath}
z\mapsto L(x,z):=\mathbb{E}_x\bigl[ e^{zT_0}\bigr]
\end{displaymath} 
has a meromorphic continuation to $\lbrace z \in \mathbb{C}\mid \Re z< \lambda_a+ \varepsilon \rbrace$.
\end{corollary}

We now aim to study the dominant poles.
\begin{proposition}\label{prop:gf}
The function $L(z)$ which meromorphically extends the Laplace transform of $T_0$ has a simple pole at
$z=\lambda_a$ and no other poles on the interval $\{z=\lambda_a+i\psi,\ \psi\in(-\pi,\pi)\}$.
\end{proposition}
\begin{proof}
We first observe that 
\begin{displaymath}
\mathbb{E}_x\bigl[z^{T_0}\bigr]=\sum_{n=0}^{\infty}\mathbb{P}_x(T=n)z^{n}
\end{displaymath}
is a power series with non-negative coefficients and therefore by Pringsheim' s theorem (see e.g. Theorem 4.1.2 in \cite{MN91}) the radius of convergence $r_a=e^{\lambda_a}$ is a singularity. Therefore $\lambda_a$ is a singularity of $L$ and as $L$ is meromorphic we conclude that $\lambda_a$ is a pole. 
Observe that this implies that the operator valued function $(I-K_z)^{-1}$ has a pole at $z=\lambda_a$. Therefore, we conclude that $K_{\lambda_a}$ has the eigenvalue $1$.

Observe next that for $\lambda<\lambda_a+\varepsilon$ the operator $K_{\lambda}$ is positive and compact on the Banach lattice $X$. Using the second part of Lemma \ref{lem:K_lambda} we conclude by the classical Krein-Rutman theorem as given in Theorem 4 of \cite{S64} that the spectral radius $r(K_{\lambda})>0$ is in fact an eigenvalue with algebraic multiplicity one and the associated eigenfunction can be chosen to be non-negative. Furthermore, by Theorem 2.1 in \cite{D08} the spectral radius is continuous in $\lambda$.

We now claim that $1$ coincides with the spectral radius of $K_{\lambda_a}$.  Assume contrary that the spectral radius $r(K_{\lambda_a})$ of $K_{\lambda_a}$ is strictly bigger then $1$. Under this assumption and the continuity of the spectral radius we conclude that for $\lambda<\lambda_a$ we still have $\rho(K_{\lambda})>1$. We now claim, that this contradicts that the equation
\begin{equation}\label{e:renewaleq}
L_\lambda=F_{\lambda}+K_{\lambda}L_\lambda
\end{equation}
holds. Observe that by Lemma \ref{l:pos-function} the function $F_{\lambda}$ is lower bounded by $1$ and iterating equation \eqref{e:renewaleq} we conclude that for $n\geq 1$ 
\begin{displaymath} 
L_{\lambda} \geq K_{\lambda}^n\mathbf{1}.
\end{displaymath}
This contradicts the fact $\|K_{\lambda}\mathbf{1}\|_{\infty} \rightarrow\infty$ if $n\rightarrow \infty$ as a consequence of the spectral radius formula $\rho(K_{\lambda})=\lim_{n\rightarrow\infty}\|K_{\lambda}^n\|^{1/n}$. 

We now aim to show that the function $L$ has a simple pole at $\lambda_a$. Here, we can observe that the derivative $\frac{d}{d\lambda}K_{\lambda}$ for real $\lambda < \lambda_a+\varepsilon$ defines a positive operator on the Banach space $X$. Therefore, we have shown, that the properties $P1$, $P2$ and $P3$ in \cite{SW75} are satisfied and, as shown on page 233 of \cite{SW75}, we can then conclude via Corollary 1 in \cite{SW75} that the pole at $\lambda_a$ is in fact simple.  

Now assume that $L(z)$ has another pole at $z=\lambda_a+i\psi$, i.e. $K_z$ has an eigenvalue 1 with some eigenfunction $g$ there. Applying the triangle inequality to the definition of $K_zg$, one obtains $|g(x)|=\left|K_{z}g(x)\right|\le K_{\lambda_a}|g(x)|$ for all $x\in[0,r]$.
The positivity of the density $\varphi$ implies that, for any $m,n\in\mathbb N$, it happens with positive probability that $\sigma_r=m,T_r=n$ and the whole expression inside the expectation value defining $K_z$ is nonzero. Consequently, its phase is not a.s. constant unless $\psi$ is a multiple of $2\pi$, the triangle inequality is even strict and, by compactness of $[0,r]$,
$$\inf_{x\in[0,r]}\frac{K_{\lambda_a}|g(x)|}{|g(x)|}>1\,.$$
Now we look at $K_{\lambda_a}$ as an operator on the real-valued continuous functions $C([0,r])$.
Since $r(K_{\lambda_a})=1$, what we just found would mean that the min-max principle in the form
\begin{equation*}
r(K_{\lambda_a}) = \max_{h\in\mathcal K}\inf_{\substack{h'\in \mathcal H' \\ \langle h,h'\rangle>0}} \frac{\langle K_{\lambda_a}h,h'\rangle}{\langle h,h'\rangle}
\end{equation*}
with $\mathcal K$ denoting the nonnegative elements of $C([0,r])$ and $\mathcal H':=\{h\mapsto h(x)\mid x\in[0,r]\}$ is violated. However, all conditions from \cite{Marek66} ensuring its validity are satisfied:
\begin{itemize}
\item $C([0,r])=\mathcal K-\mathcal K$ and the cone $\mathcal K$ is closed and normal ($\|g+h\|\ge\|g\|$ for all normalised $g,h\in\mathcal K$).
\item $\mathcal H'$ is total, i.e. if $\langle h,h'\rangle\ge0$ for all $h'\in\mathcal H'$, then $h\in\mathcal K$.
\item $K_{\lambda_a}$ is semi non-supporting: For all nonzero $h\in\mathcal K$,  $K_{\lambda_a}h(x)>0$ for all $x$ as in Lemma~\ref{lem:K_lambda} and therefore one can conclude for all nonzero $h'\in\mathcal K'$ (the dual cone, i.e. positive measures) $\langle K_{\lambda_a}h,h'\rangle>0$.
\item The resolvent $\lambda\mapsto(\lambda-K_{\lambda_a})^{-1}$ has only finitely many singularities with $|\lambda|=r(K_{\lambda_a})$, all of them poles, because $K_{\lambda_a}$ is compact.
\end{itemize}
\end{proof}
\begin{proof}[Completion of the proof of Theorem~\ref{thm.analytic}]
It follows from the Proposition~\ref{prop:gf} that the generating function $\mathbb{E}_x[z^{T_0}]$
has the unique simple pole at $e^{\lambda_a}$ and that there are no further
poles on the disc or radius $e^{\lambda_a+\gamma}$ with some $\gamma>0$. Then we have the following
representation
$$
\mathbb{E}_x[z^{T_0}]=\frac{V(x)}{e^{\lambda_a}-z}+g_x(z),
$$
where $g_x$ is analytical on the disc with radius $e^{\lambda_a+\gamma}$. Therefore,
\begin{align}\label{thm.an.1wdh}
\mathbb{P}_x(T_0=n)=V(x)e^{-\lambda_a(n+1)}+O(e^{-(\lambda_a+\gamma)n}). 
\end{align}
Since the left hand side is positive, we infer that the function $V$ is positive as well.
Therefore, it remains to show that $v$ is $e^{\lambda_a}$-harmonic for $P_+$. Let $r$ be so large that
$\mathbb{E}_y[e^{\lambda_a T_r}]$ is finite. For every $x\ge r$ one has the inequality
$$
\mathbb{P}_x(T_0\ge n)\le\sum_{m=1}^n\mathbb{P}_x(T_r=m)\mathbb{P}_r(T_0\ge n-m).
$$
It follows from \eqref{thm.an.1wdh} that $\mathbb{P}_r(T_0\ge k)\le C(r)e^{-\lambda_a(k+1)}$, $k\ge0$. Thus,
\begin{equation}
\label{thm.an.2}
\mathbb{P}_x(T_0\ge n)\le C(r)e^{-\lambda_a n}\sum_{m=1}^n\mathbb{P}_x(T_r=m)
\le C(r)\mathbb{E}_x[e^{\lambda_a T_r}]e^{-\lambda_a(n+1)}.
\end{equation}
It is immediate from \eqref{thm.an.1wdh} that
\begin{equation}
\label{thm.an.3}
\mathbb{P}_x(T_0\ge n)= \frac{V(x)}{1-e^{-\lambda_a}}e^{-\lambda_a(n+1)}
+O\left(e^{-(\lambda_a+\gamma)n}\right).
\end{equation}
From this equality and from \eqref{thm.an.2} we infer that
\begin{equation}
\label{thm.an.4}
V(x)\le \frac{C(r)}{1-e^{-\lambda_a}}\mathbb{E}_x[e^{\lambda_a T_r}],\quad x\ge r.
\end{equation}
Since
$$
\mathbb{E}_x[e^{\lambda_a T_r}]=\int_0^r P(x,dy)e^{\lambda_a}+
\int_r^\infty P(x,dy)e^{\lambda_a}\mathbb{E}_y[e^{\lambda_a T_r}],
$$
we infer that
\begin{equation}
\label{thm.an.5}
\int_r^\infty P(x,dy)\mathbb{E}_y[e^{\lambda_a T_r}]<\infty
\end{equation}
and, in view of \eqref{thm.an.4},
\begin{equation}
\label{thm.an.6}
\int_0^\infty P(x,dy)V(y)<\infty.
\end{equation}

Fix some $A>r$ and consider the equality
\begin{align}
\label{thm.an.7}
\mathbb{P}_x(T_0\ge n+1)&=\int_0^\infty P(x,dy)\mathbb{P}_y(T_0\ge n)\\
\nonumber
&=\int_0^A P(x,dy)\mathbb{P}_y(T_0\ge n)+\int_A^\infty P(x,dy)\mathbb{P}_y(T_0\ge n).
\end{align}
Combining \eqref{thm.an.2} and \eqref{thm.an.5}, we obtain
\begin{equation}
\label{thm.an.8}
\lim_{A\to\infty}\limsup_{n\to\infty}
e^{\lambda_a(n+1)}\int_A^\infty P(x,dy)\mathbb{P}_y(T_0\ge n)=0.
\end{equation}
Furthermore, by \eqref{thm.an.3} and \eqref{thm.an.6},
\begin{equation}
\label{thm.an.9}
\lim_{A\to\infty}\lim_{n\to\infty}e^{\lambda_a(n+1)}\int_0^A P(x,dy)\mathbb{P}_y(T_0\ge n)
=\int_0^\infty P(x,dy)V(y).
\end{equation}
Plugging \eqref{thm.an.8} and \eqref{thm.an.9} into \eqref{thm.an.7}, we obtain
$$
\lim_{n\to\infty} e^{\lambda_a(n+1)}\mathbb{P}_x(T_0\ge n+1)=\int_0^\infty P(x,dy)V(y).
$$
According to \eqref{thm.an.3}, the limit on the left hand side equals to $e^{-\lambda_a}V(x)$.
As a result we have the equality
$$
e^{-\lambda_a}V(x)=\int_0^\infty P(x,dy)V(y).
$$
Therefore, the proof is complete.
\end{proof}
%%%%%%%%%%%%%%%%%%%%%%%%%%%%%%%%%%%%%%%%%%%%%%%%%%%%%%%%%%%%%%%%%%%%%%%%%%%%%%%%%%%%%%%%%%%%%%%%%%%%%%%%
%%%%%%%%%%%%%%%%%%%%%%%%%%%%%%%%%%%%%%%%%%%%%%%%%%%%%%%%%%%%%%%%%%%%%%%%%%%%%%%%%%%%%%%%%%%%%%%%%%%%%%%%
\section{Innovations with regularly varying tails}\label{s:reg}
The main purpose of this section is to show that the finiteness of all moments of $\xi_1^+$ is
necessary for getting purely exponential decay for the tail of $T_0$. More precisely, we are going to
show that if the right tail is regulaly varying then, independent of the index of regular variation, the 
asymptotic behaviour of $\mathbb{P}_x(T_0>n)$ depends on the slowly varying component of
$\mathbb{P}(\xi_1>x)$. In particular, it may happen that $e^{\lambda_a n}\mathbb{P}_x(T_0>n)\to0$ as
$n\to\infty$.
\begin{proposition}
\label{prop:reg.tails}
Assume that 
\begin{equation}
\label{reg.tails.0}
\mathbb{P}(\xi_1>x)=x^{-r}L(x)
\end{equation}
for some $r>0$ and some slowly varying function $L$.
Then, for all $M\ge0$,
\begin{equation}
\label{reg.tails.1}
\liminf_{n\to\infty}\frac{1}{n}\log\mathbb{P}(T_M>n)\ge r\log a.
\end{equation}
If, in addition, 
\begin{equation}
\label{reg.tails.1a}
L(x)=O(\log^{-2r-2} x),
\end{equation}
then, for all $M$ sufficiently large, 
\begin{equation}
\label{reg.tails.2}
\lim_{n\to\infty} a^{-rn}\mathbb{P}(T_M>n)=0.
\end{equation}
In fact, one even has 
\begin{displaymath}
\sum_{n=0}^{\infty}a^{-rn}\mathbb{P}(T_M>n)<\infty.
\end{displaymath}
\end{proposition}
\begin{remark}
It is immediate from \eqref{reg.tails.1} and \eqref{reg.tails.2} that
$$
\lambda_a=-r\log a
$$
for innovations satisfying \eqref{reg.tails.0} and \eqref{reg.tails.1a}. We conjecture that the same holds
under the assumption \eqref{reg.tails.0} and that the asymptotic behavior of the slowly varying
function $L$ affects lower order corrections only.
\hfill$\diamond$
\end{remark}
\begin{proof}[Proof of Proposition~\ref{prop:reg.tails}.]
Fix some $A>a^{-1}$ and define $u_k=MA^k$, $k\ge0$. Then
\begin{align*}
\mathbb{P}_{u_k}(X_1\ge u_{k-1})=\mathbb{P}(au_k+\xi_1>u_{k-1})=\mathbb{P}(\xi_1>-(aA-1)u_{k-1}). 
\end{align*}
Therefore, using the Markov property and the stochastic monotonicity of $X_n$, we get
\begin{align*}
\mathbb{P}_{u_k}(T_M>k-1)
&\ge\mathbb{P}_{u_k}(X_1\ge u_{k-1},X_2>u_{k-2},\ldots,X_{k-1}\ge u_1)\\
&=\prod_{j=2}^{k}\mathbb{P}_{u_j}(X_1\ge u_{j-1})=\prod_{j=1}^{k-1}\mathbb{P}(\xi_1>-(aA-1)u_{j}).
\end{align*}
It follows now from the assumption $\mathbb{E}\log(1+|\xi_1|)<\infty$ that
$$
\inf_{k\ge2}\prod_{j=1}^{k-1}\mathbb{P}(\xi_1>-(aA-1)u_{j})=:p(M,A)>0.
$$
This implies that
$$
\mathbb{P}_M(T_M>k)\ge p(M,A)\mathbb{P}(\xi_1>u_k).
$$
Since the tail of $\xi_1$ is regularly varying of index $-r$, we obtain
$$
\liminf_{k\to\infty}\frac{1}{k}\log\mathbb{P}_M(T_M>k)\ge -r\lim_{k\to\infty}\log u_k
=-r\log A.
$$
Letting now $A\downarrow a^{-1}$, we arrive at \eqref{reg.tails.1}.

In order to prove the second statement we consider events
$$
B_n:=\left\{\xi_k\le h\frac{a^{-(n-k+1)}}{(n-k+1)^2}\text{ for all }k<T_M\wedge n\right\}.
$$
On the event $\{T_M>n\}\cap B_n$ one has
\begin{align*}
X_n=a^nM+\sum_{k=1}^na^{n-k}\xi_k\le a^n M+\frac{h}{a}\sum_{k=1}^n\frac{1}{(n-k+1)^2}
\le aM+\frac{2h}{a}.
\end{align*}
Thus, for every $M\ge \frac{2}{a(1-a)}h$ and all $n\ge1$ one has $X_n\le M$ and,
consequently, $\{T_M>n\}\cap B_n=\emptyset$. This implies that
$$
\mathbb{P}_M(T_M>n)=\mathbb{P}_M(T_M>n;B_n^c)
\le\sum_{k=1}^n\mathbb{P}_M(T_M>k-1)\mathbb{P}\left(\xi_k>h\frac{a^{-(n-k+1)}}{(n-k+1)^2}\right).
$$
From the assumption~\ref{reg.tails.1a} on $L(x)$ we get
$$
\mathbb{P}\left(\xi_k>h\frac{a^{-(n-k+1)}}{(n-k+1)^2}\right)
\le\frac{c(a)}{h^r}\frac{a^{r(n-k+1)}}{(n-k+1)^2}.
$$
Therefore,
\begin{align*}
\mathbb{P}_M(T_M>n)\le\frac{c(a)}{h^r}\sum_{j=0}^{n-1}\mathbb{P}_M(T_M>j)\frac{a^{r(n-j)}}{(n-j)^2},
\quad n\ge1.
\end{align*}
Multiplying both sides with $s^n$ and summing over all $n$, we obtain
\begin{align*}
\sum_{n=0}^\infty s^n\mathbb{P}_M(T_M>n)
&\le 1+\sum_{n=1}^\infty s^n \frac{c(a)}{h^r}\sum_{j=0}^{n-1}\mathbb{P}_M(T_M>j)\frac{a^{r(n-j)}}{(n-j)^2}\\
&=1+\frac{c(a)}{h^r}\sum_{j=0}^\infty s^j\mathbb{P}_M(T_M>j)\sum_{n=j+1}^\infty\frac{(sa^r)^{n-j}}{(n-j)^2}\\
&=1+\frac{c(a)}{h^r}\sum_{j=0}^\infty s^j\mathbb{P}_M(T_M>j)\sum_{n=1}^\infty \frac{(sa^r)^{n}}{n^2}.
\end{align*}
In other words,
\begin{align*}
\sum_{n=0}^\infty s^n\mathbb{P}_M(T_M>n)\le 
\left(1-\frac{c(a)}{h^r}\sum_{n=1}^\infty \frac{(sa^r)^{n}}{n^2}\right)^{-1}.
\end{align*}
If $h$ is so large that $h^r>4c(a)$ then we have
\begin{align*}
\sum_{n=0}^\infty a^{-rn}\mathbb{P}_M(T_M>n)\le 2.
\end{align*}
This yields \eqref{reg.tails.2}.
\end{proof}
In order to see, that spectral properties remain relevant in the situation at hand we observe the following assertion formulated in terms of Tweedie's \textbf{R}-theory (see \cite{T74a} and \cite{T74b}).
\begin{corollary}
Assume that the conditions \eqref{reg.tails.0} and \eqref{reg.tails.1a} are satisfied. Then for $r$ large enough the operator $K_{\lambda_a}$ is well defined and the spectral radius $r(K_{\lambda_a})$ belongs to $[0,1]$. Under the assumption $r(K_{\lambda_a})<1$ the Laplace transform of $T_0$ remains bounded up to the critical line. In particular, the Submarkovian transition operator $P$ is $R$-transient in the sense of Tweedie. 
\end{corollary}
Until now a complete analysis of persistence probabilities including the effects of polynomial decay
factors as well as thorough investigation the quasistationary behaviour of AR(1) sequences with heavy
tailed innovations does not seem to exist and constitutes an interesting open problem.

\section{Discussion}\label{s:dis}
In this section we summarise the general ideas of the two main analytic approaches used in this work and comment on their possible applicability to other models. As a matter of fact both approaches use different tools but also share some structural similarites. For a Markov process $(X_n)_{n \in \mathbb{N}_0}$ we want to find the precise tail behaviour of the first hitting time $T_B$ of a measurable subset $B$ of the state space.  
\begin{itemize}
\item The first approach is in essence spectral theoretic. We first analyse the decay rate
\begin{displaymath}
\theta_B=-\lim_{n\rightarrow\infty}\frac{1}{n}\log \mathbb{P}_x\bigl(T_B> n\bigr)
\end{displaymath}
and make sure that $\theta_B$ does not depend on the starting point. The second step consists in showing that for a strictly larger measurable set $B'\supset B$ one has
\begin{equation}\label{e:discussion}
\theta_B < \theta_{B'}\,.
\end{equation}
This allows to introduce a suitable weighted Banach space and under appropriate conditions on the distribution of the innovations to prove the quasicompactness of the killed transition kernel. Application of a suitable result of Perron-Frobenius type allows to establish precise exponential decay of the tails of $T_B$. 
\item 
The second approach consists in analysing the Laplace transform 
\begin{displaymath}
 \lbrace z \in \mathbb{C} \mid \Re z<\theta_B\rbrace \ni \lambda \mapsto F_\lambda(x):= \mathbb{E}_x\bigl[e^{\lambda T_B}\bigr]
\end{displaymath}
near the abscissa of convergence $\theta_B$. We prove that $F_{\lambda}(x)$ has a meromorphic extension to $\lbrace z \in \mathbb{C} \mid \Re z<\theta_B+\varepsilon\rbrace $ for some $\varepsilon>0$ with $\lambda_B$ being a pole. This allows to deduce the precise exponential decay using a suitable Tauberian theorem. In order to prove the existence of a meromorphic extension we derive a renewal type equation in terms of the transition kernel $K_\lambda$. The existence of a meromorphic extension is shown using the analytic Fredholm alternative and suitable properties of $K_{\lambda}$. In particular, we need to show that the operators $K_{\lambda}$ are compact for all $\lambda$ with $\Re \lambda < \lambda_a+\varepsilon$ and satisfy the conditions of a suitable Perron-Frobenius theorem. Here results of the type given in \eqref{e:discussion} are again essential as well as absolute continuity and strict positivity of the transition kernel. 
\end{itemize}
We believe that these methods can be applied without big changes to some other Markov chain models. For example, to max-autoregressive processes or to random exchange processes, which are quite closely related to AR($1$)-sequences, see \cite{Zerner18}.

In both approaches, we have assumed that the underlying distributions are absolutely continuous with almost everywhere positive density and have sufficiently light tails. 
The positivity on the whole axis has been used only to give very simple proofs of the posivity of compact operators. We believe that this positivity can be shown also in the case when the density is zero on significant parts of the axis. Probably one will have the positivity of an appropriate power of the operator, but this does not restrict the applicability of Perron-Frobenius type results. The absolute continuity assumption seems to be really crucial for both approaches, and it is not clear whether one can replace it by the existence of an absolute continuous component. A very challenging problem is to study the case
of discrete innovations.

The second approach might also be applicable to chains with a certain periodicity, presumably resulting in poles of higher order. A different behaviour of the pole(s) also can be expected in the presence of distributions with regularly varying tails. Getting completely rid of absolute continuity seems so be much more tricky. Moreover, the stochastic monotonicity of AR(1)-sequences was extensively used, a lack of this property is expected to result in largely technical complications.

The authors of \cite{CV17} are able to prove similar results using a different approach. Their conditions are of the following types (see page 8 in \cite{CV17}):
\begin{itemize}
\item Local minorization of Doeblin-type
\item Two global Lyapunov criteria
\item Local Harnack inequality
\item Aperiodicity
\end{itemize}
The approach used by Champagnat and Villemonais is more related to coupling ideas of Doeblin type. The global Lypunov criteria are related to our condition \eqref{e:discussion} whereas the local minorization, the local Harnack inequality as well as the aperiodicity are closer to the condition used in our approaches to prove quasicompactness with a leading eigenvalue of mulitiplicity one in approach one and compactness of the renewal operator with a leading eigenvalue of multiplicity one in approach two.  All three approaches have their merits. As far as possible extensions to AR(1)-processes with fat tailed innovations are concerned, the approach in Section \ref{s:renewal} seems to be most promising. 
\vspace{6pt}

\textbf{Acknowledgement.} The authors would like to thank a referee for his useful report and for suggestion to add the final section. This project has been initiated during the subbatical stay of VW at the Technion, Israel. He thanks the Humboldt foundation and the Technion for the financial support. MK would like to thank his student P. Trykacz for critical reading of the final draft.
%%%%%%%%%%%%%%%%%%%%%%%%%%%%%%%%%%%%%%%%%%%%%%%%%%%%%%%%%%%%%%%%%%%%%%%%%%%%%%%%%%%%%%%%%%%%%%%%%%%%%%%%
%%%%%%%%%%%%%%%%%%%%%%%%%%%%%%%%%%%%%%%%%%%%%%%%%%%%%%%%%%%%%%%%%%%%%%%%%%%%%%%%%%%%%%%%%%%%%%%%%%%%%%%%

\end{document}